\documentclass[a4paper,12pt]{article}
\usepackage{amsthm,amsmath,amssymb,txfonts,graphics,mathrsfs}
\usepackage[latin1]{inputenc}

\usepackage{color}
\usepackage{pdfcolmk}
\usepackage{pb-diagram}
\usepackage{hyperref}
\newtheorem{defin}{Definition}
\newtheorem{lemma}{Lemma}
\newtheorem{teo}{Theorem}
\newtheorem{oss}{Remark}
\newtheorem{prop}{Proposition}

\newtheorem{example}{Example}

\author{Jung Kyu CANCI}
\title{Rational periodic points for quadratic maps.
}
\date{}
\newcommand{\nc}{\newcommand}
\nc{\pro}{\mathbb{P}_1}
\nc{\po}{\mathbb{P}_1(K)}
\nc{\pq}{\mathbb{P}_1(\mathbb{Q})}
\nc{\ciclon}{(P_0, P_1,\ldots,P_{n-1})}
\nc{\orbit}{(P_{-m},\ldots,P_{-1},P_0,\ldots,P_{n-1})}
\nc{\cicloP}{\{P_i\}}
\nc{\tildep}{$\widetilde{a}^{\hspace{-0.65mm}\phantom{\tiny.}^{\tiny p}}$}
\nc{\spazio}{\hspace{0.7mm}}
\nc{\bdm}{\begin{displaymath}}
\nc{\edm}{\end{displaymath}}
\nc{\beq}{\begin{equation}}
\nc{\eeq}{\end{equation}}
\nc{\ind}{\hspace{-1.3mm}{\phantom{\_}}_}
\nc{\got}{\sffamily{p}\normalfont}
\nc{\ri}{R_S}
\nc{\ru}{R_S^\ast}
\nc{\qc}{\emph{quasi coprime}}
\nc{\p}{\mathfrak{p}}
\nc{\vvs}{\textquotedblleft}
\nc{\vvd}{\textquotedblright}
\nc{\dip}{\delta_{\p}}
\nc{\vap}{v_{\p}}
\nc{\K}{\mathbb{K}}
\nc{\bpm}{\begin{pmatrix}}
\nc{\epm}{\end{pmatrix}}
\nc{\RR}{\mathbb{R}}
\nc{\ZZ}{\mathbb{Z}}
\nc{\NN}{\mathbb{N}}
\nc{\FF}{\mathbb{F}}
\nc{\KK}{\mathbb{K}}
\nc{\ds}{\mathbb{S}}
\nc{\da}{\mathbb{A}}
\nc{\Ts}{T_{\ds}}
\nc{\ov}{\overline}
\nc{\lra}{\Leftrightarrow}

\begin{document}
\maketitle
\begin{abstract}
Let $K$ be a number field. Let $S$ be a finite set of places of $K$ containing all the archimedean ones. Let $R_S$ be the ring of $S$-integers of $K$. In the present paper we consider endomorphisms of $\pro$ of degree $2$, defined over $K$, with good reduction outside $S$. We prove that there exist only finitely many such endomorphisms, up to conjugation by ${\rm PGL}_2(R_S)$, admitting a periodic point in $\po$ of order $>3$. Also, all but finitely many classes with a periodic point in $\po$ of order $3$ are parametrized by an irreducible curve.
\end{abstract}

\section{Introduction.} \noindent For an integer $d\geq 1$, let ${\rm Rat}_d$ denote the space of all rational maps (or morphisms) $\pro \to \pro$ of degree $d$. Every morphism $\Phi\in{\rm Rat}_d$, defined over a field $K$, is given by a pair of homogeneous polynomials $F,G\in K[X,Y]$ in the following way
\begin{align} \label{Phi}\Phi([X:Y])=&[F(X,Y):G(X,Y)]\\\notag=&[f_0X^d+f_1X^{d-1}Y+\ldots+f_dY^d:g_0X^d+g_1X^{d-1}Y+\ldots+g_dY^d];\end{align}
we suppose that $F,G$ have no common roots in $\pro(\overline{K})$, where $\overline{K}$ denotes the algebraic closure of $K$. This last requirement is equivalent to the condition that the resultant ${\rm Res}(F,G)$ of the polynomials $F$ and $G$ is not zero. The resultant ${\rm Res}(F,G)$ is a certain bihomogeneous polynomial in the coefficients $f_0,f_1,\ldots,f_d,g_0,\ldots,g_d$ (e.g. see \cite[Chapter IV.8]{L.2}) so that ${\rm Rat}_d$ is parametrized by the Zariski open subset of $\mathbb{P}_{2d+1}$ defined by ${\rm Res}(F,G)\neq 0$.

The group ${\rm GL}_2$ induces the following conjugation action on ${\rm Rat}_d$: let $A=\begin{pmatrix}\alpha&\beta\\ \gamma & \delta\end{pmatrix}\in {\rm GL}_2$ and let $[A]$ be the canonical automorphism of $\pro$ associated to $A$; for every map $\Phi=[F:G]\in{\rm Rat}_d$ we have \begin{align*}[A]\circ\Phi\circ[A]^{-1}=[&\alpha F(\delta X-\beta Y,-\gamma X+\alpha Y)+\beta G(\delta X-\beta Y,-\gamma X+\alpha Y) :  \\ &\gamma F(\delta X-\beta Y,-\gamma X+\alpha Y)+\delta G(\delta X-\beta Y,-\gamma X+\alpha Y) ].\end{align*} 
The action factors through ${\rm PGL}_2={\rm GL}_2/\{\lambda I_2\mid \lambda\in\mathbb{G}_m\}$, where $I_2=$\footnotesize{$\bpm 1&0\\0&1\epm$}\normalsize, and we set $M_d={\rm Rat}_d/{\rm PGL}_2$.

J. Milnor in \cite{Mil.1} proved that the moduli space $M_2(\mathbb{C})={\rm Rat}_2(\mathbb{C})/{\rm PGL}_2(\mathbb{C})$ is analytically isomorphic to $\mathbb{C}^2$. 
J.H. Silverman in \cite{Sil.1} generalized this result. For various reasons, he considered the action of SL$_2$ and he proved that the quotient space ${\rm Rat}_d/{\rm SL}_2$ exists as a geometric quotient scheme over $\mathbb{Z}$ in the sense of Mumford's geometric invariant theory and that ${\rm Rat}_2/{\rm SL}_2$ is isomorphic to $\mathbb{A}^2_{\mathbb{Z}}$. Note that if $\Omega$ is an algebraically closed field, then we have a projection SL$_2(\Omega)\to{\rm PGL}_2(\Omega)$ which is surjective with kernel $\{\pm 1\}$.

Throughout all the present paper, $K$ denotes a number field with ring of integers $R$. Let $S$ be a fixed finite set of places of $K$ containing all the archimedean ones. Following the definition used in \cite{Sil.2}, an integer $n$ is called the exact (or primitive) period of a point $P$ for a map $\Phi$ if $n$ is the minimal positive integer such that $\Phi^{n}(P)=P$, where $\Phi^n$ denotes the $n$-th iterate of $\Phi$.
Given a positive integer $n$, we are interested in studying the subset of ${\rm Rat}_2(\overline{\mathbb{Q}})/{\rm PGL}_2(\overline{\mathbb{Q}})$ of classes containing a rational map, defined over $K$ with good reduction outside $S$, which admits a periodic point in $\po$ with exact period equal to $n$. 
The notion of good reduction used in the present paper is the following one:  Let   $\Phi\colon\mathbb{P}_1\to\mathbb{P}_1$ be a rational map defined over $K$ as in (\ref{Phi}). Let $\p$ be a non zero prime ideal of $R$, $R_{\p}$ the local ring of $R$ at the prime ideal $\p$ and let $K(\mathfrak{p})=R/\mathfrak{p}$ be the residue field. Suppose that $F,G$ have coefficients in $R_{\mathfrak{p}}$ and that at least one coefficient is a unit in $R_{\mathfrak{p}}$. Let $\tilde{F}$ and $\tilde{G}$ be the polynomials obtained by reducing modulo $\mathfrak{p}$ the coefficients of $F$ and $G$. We shall say that $\Phi$ has good reduction at the prime ideal $\mathfrak{p}$ if the polynomials $\tilde{F}$ and $\tilde{G}$ do not have common roots in $\mathbb{P}_1(\overline{K(\mathfrak{p})})$, where $\overline{K(\mathfrak{p})}$ denotes the algebraic closure of $K(\mathfrak{p})$. Note that, by the assumption on the coefficients of $F$ and $G$, the rational map $\Phi$ has good reduction at the prime ideal $\p$ if the resultant of $F$ and $G$ is a $\p$-unit. We say that a rational map has good reduction outside $S$ if it has good reduction at every prime ideal $\p\notin S$. 

We denote by $R_S$ the ring of $S$-integers and by $R_S^*$ the group of $S$-units (see the next section for the definitions).

The group ${\rm PGL}_2(R_S)$, quotient of ${\rm GL}_2(R_S)$ modulo scalar matrices, acts on $\pro$. It is the group of automorphisms having good reduction outside $S$. It also acts by conjugation on the set of rational maps, defined over $K$ with good reduction outside $S$. Indeed, given an automorphism $[A]\in {\rm PGL}_2(R_S)$ and a rational map $\Phi$ defined over $K$, the rational map $[A]\circ\Phi\circ [A]^{-1}$ has good reduction outside $S$ if and only if the map $\Phi$ has good reduction outside $S$. Furthermore, it is clear that if the $n$-tuple $\ciclon$ is a cycle for the rational map $\Phi$, then the $n$-tuple $([A](P_0),[A](P_1),\ldots,[A](P_{n-1}))$ is a cycle for the rational map $[A]\circ\Phi\circ [A]^{-1}$. 
  
\begin{defin}\label{eqcy}We say that two $n$-tuples $(P_0,P_1,\ldots,P_{n-1})$ and $(Q_0,Q_1,\ldots,Q_{n-1})$ in $\po$ are equivalent if there exists an automorphism $[A]\in{\rm PGL}_2(R_S)$ and an integer $0\leq h\leq n-1$ such that $[A](P_i)=Q_{i+h}$ for every index $0\leq i\leq (n-1)$. For every index $t\geq n$, $Q_t$ denotes the point $Q_{\bar{t}}\in \{Q_0,Q_1,\ldots,Q_{n-1}\}$, where the integer $\bar{t}$ is such that $0\leq\bar{t}\leq n-1$ and $\bar{t}\equiv t$ (mod $n$).
\end{defin}

By the above arguments, we are interested in the action of ${\rm PGL}_2(R_S)$ rather than that of ${\rm PGL}_2(\overline{\mathbb{Q}})$. In this paper we shall prove the following result:

 {\begin{teo}\label{tnf}
Let $n\geq 4$ be a fixed integer. Then with respect to the action of ${\rm PGL}_2(R_S)$, there are only finitely many conjugacy classes of quadratic rational maps defined over $K$ having good reduction outside $S$ and admitting a periodic point in $\po$ with exact period $n$.
\end{teo}

By using a result obtained by Morton and Silverman it is possible to improve on this statement. Indeed Morton and Silverman proved in \cite{M.S.1} that the maximal length of a cycle, for an endomorphism $\Phi$ defined over $K$ with good reduction outside $S$, can be bounded in terms of the cardinality of the set $S$ only. Therefore Theorem \ref{tnf} becomes

\bigskip
 \noindent {\bf Theorem 1$^\prime$.} \emph{With respect to the action of ${\rm PGL}_2(R_S)$, there are only finitely many conjugacy classes of quadratic rational maps defined over $K$ having good reduction outside $S$ and admitting a periodic point in $\po$ with exact period $\geq 4$.}  

\bigskip

From Theorem \ref{tnf}$^\prime$ it follows that there are only finitely many inequivalent $n$-tuples in $\po$ with $n\geq 4$ which are cycles for a rational map defined over $K$ with good reduction outside $S$. If we dropped the condition on good reduction, Theorem \ref{tnf} would be false, as the following example shows:
\begin{example}\label{engr} \emph{For every non zero rational number $a\notin\{-2,2,4\}$, the following ordered $4$-tuple 
\bdm ([0:1], [a:1], [1:0], [2:1])\edm
is a cycle for the quadratic rational map $\Phi_a([X:Y])=[(X-2Y)(4X-a^2Y):2(X-aY)(X-Y)]$. Let $n$ and $d$ be coprime integers such that $a=n/d$, then the rational map $\Phi_a$ has bad reduction at every prime ideal that divides $2\cdot d\cdot n(n-4d)(n^2-4d^2)$. Furthermore, it is easy to see that, given a fixed non zero rational number $a\notin\{-2,2,4\}$, there exist only two rational numbers $b$ such that the two ordered $4$-tuples $([0:1], [a:1], [1:0], [2:1])$ and $([0:1], [b:1], [1:0], [2:1])$ are equivalent.}
\end{example}

 \begin{oss}\emph{The previous example describes an infinite family of quadratic rational maps defined over $\mathbb{Q}$, in which every element has a $\mathbb{Q}$-periodic point of exact period $4$. It would be interesting to study the integers $n\neq 4$ such that all (or infinitely many) elements of the family in Example \ref{engr} also have a $\mathbb{Q}$-periodic point of exact period $n$. For example, in the case $n=2$ it is possible to see that there exist infinitely many values of $a\in \mathbb{Q}$ such that the rational map $\Phi_a$ also admits a $\mathbb{Q}$-periodic point of exact period $2$; we give a very brief proof of this fact, since it is beyond the scope of this article. Let $(a,t)\in\mathbb{Q}^2$ such that $a\notin\{-2,0,2,4\}$. If the point $[t:1]\in\pro(\mathbb{Q})$ is a periodic point of exact period $2$ for the rational map $\Phi_a$, then the pair $(a,t)$ is a $\mathbb{Q}$-rational point of the curve $\mathcal{C}$ defined by $2t^2-a^2t+3at-4t+a^2-4a=0$. The converse is not in general true, but if the curve $\mathcal{C}$ has infinitely many $\mathbb{Q}$-rational points, then there exist infinitely many values of $a$ such that the rational map $\Phi_a$ admits a $\mathbb{Q}$-periodic point of exact period $2$. Some calculations tell us that the curve $\mathcal{C}$ is an affine piece of an elliptic curve $E$ where the Mordell-Weil rank of $E(\mathbb{Q})$ is positive. Therefore the curve $\mathcal{C}$ has infinitely many $\mathbb{Q}$-rational points. }
\end{oss}  

The main tool for the proof of Theorem \ref{tnf} is the $S$-Unit Equation Theorem (See \cite{E.3}\cite{VDP.1}\cite{Sc.2}). Other tools are some divisibility arguments, such as Proposition  6.1    provided by Morton and Silverman in \cite{M.S.2}, and the finiteness result on cycles proved in \cite[Theorem 1]{C.1}. Since the $S$-Unit Equation Theorem rests on the ineffective Subspace Theorem (see \cite{Sc.2},\cite{Sc.1}), Theorem \ref{tnf} is also ineffective.

A section of this paper is dedicated to the case of periodic points of exact period $n=3$, more precisely, the case in which a class of quadratic rational maps contains a map defined over $K$ with good reduction outside $S$, which admits a periodic point in $\po$ with exact period $n=3$. We have dedicated an entire section to this case since it makes use of extra techniques. The main tools used in this section are some results proved by Corvaja and Zannier in \cite{C.Z.1} in which they obtained the estimate, for every fixed $\epsilon>0$, ${\rm gcd}(u-1,v-1)<\max(H(u),H(v))^{\epsilon}$ for all but finitely many multiplicatively independent $S$-units (where {\bf gcd} denotes a suitable notion of greatest common divisor on number fields and $H(\cdot)$ denotes the multiplicative height). The results proved by Corvaja and Zannier in \cite{C.Z.1} also rest on the Subspace Theorem. 

Our result in this case is:

\begin{teo}\label{m1} All but finitely many conjugacy classes of quadratic rational maps defined over $K$, with respect to the action of ${\rm PGL}_2(R_S)$, having good reduction outside $S$ and admitting a periodic point in $\po$ with exact period $3$, are representable by a rational map of the form
\beq\label{param} \Phi([X:Y])=[(X-Y)(aX+Y):aX^2]\eeq
where $a\in R_S^*$.
\end{teo}

At the end of the present paper we shall prove that the maps of the family (\ref{param}) represent an infinite set of elements in ${\rm Rat}_2(\overline{\mathbb{Q}})/{\rm PGL}_2(\overline{\mathbb{Q}})$.

Recall that, by Silverman's result in \cite{Sil.1}, the moduli space $M_2$ is an affine surface isomorphic to $\mathbb{A}^2$. Theorem \ref{m1} states that the set of elements in $M_2$, containing a map defined over $K$ with good reduction outside $S$ admitting a periodic point in $\po$ with exact period $3$, is not Zariski dense. 
More precisely, it consists  of a finite set, depending on $S$, plus an infinite family contained in a curve parametrized by $\mathbb{G}_m$. It is the family parametrized by (\ref{param}), and its elements, as stated in Theorem \ref{m1}, have a point of period $3$.

 Theorem \ref{tnf} and Theorem \ref{m1} also provide the following geometric interpretation. Since char$(K)=0$, following Milnor's results (\cite{Mil.1}), we see that every class in $M_2$ contains a rational function of the following form 
\beq\label{MNF} f(z)=z+\dfrac{1}{z}\ \ \ \text{or}\ \ \ f_{b,c}(z)=\dfrac{z^2+bz}{cz+1}.\eeq
A rational map of the form $f_{b,c}$ has degree two if and only if $bc\neq 1$; moreover for every given integer $n$ the 
%
condition that a point $x\in\da^1$ is a periodic point of order $n$ for the rational function $f_{b,c}$ is an algebraic condition. Therefore the set of triples $(b,c,x)\in \mathbb{A}^3$, where $bc\neq 1$ and $x$ is a periodic point of exact period $n$ for the rational function $f_{b,c}$, is a quasi projective surface $V_n$. If $(b,c,x)$ is a $S$-integral point of $V_n$, then the associate map $f_{b,c}$ has good reduction outside $S$. By elementary calculations, it is possible to see that every class in $M_2$ contains at most six rational functions of the form as in (\ref{MNF}). Moreover the number of periodic points of exact period $n$ for a quadratic rational map is bounded by a constant which depends only on $n$. In general for a rational map of degree $d$ this bound is $d+1$ for $n=1$ and $\sum_{k|n}\mu\left(\frac{n}{k}\right)d^k$ for $n>1$, where $\mu$ is the M\"obius function (see Exercise 4.3 in \cite{Sil.2}). Let $n>3$ be a given integer. By the above remarks, if $V_n$ had infinitely many $S$-integral points, there would be infinitely many classes in $M_2$ containing a rational map defined over $K$ with good reduction outside $S$ admitting a periodic point in $\po$ with exact period $n$. Therefore by Theorem \ref{tnf} the surface $V_n$ has only finitely many $S$-integral points. In a similar way we see that from Theorem \ref{m1} it follows that the set of $S$-integral points of $V_3$ is not Zariski dense.

This paper also contains a finiteness result on classes of cycles (see Definition \ref{eqcy}) for quadratic rational maps with good reduction outside $S$.  
\begin{teo}\label{tf32}The set of classes modulo PGL$_2(R_S)$ of cycles in $\po$ of length $\geq 3$, for quadratic rational maps defined over $K$ with good reduction outside $S$, is finite.
\end{teo}

The statement of Theorem \ref{tf32} would be false if we considered cycles of length $2$. In the next section we shall give an infinity family of inequivalent cycles (under the action induced by ${\rm PGL}_2(R_S)$) of length $2$ for monic polynomial maps of degree $2$ with coefficients in $R_S$. These maps have good reduction outside $S$.

Theorem \ref{tf32} is a new and non trivial result because in Definition \ref{eqcy}, about equivalence relation on cycles, we consider the action induced by ${\rm PGL}_2(R_S)$. Since the canonical action of ${\rm PGL}_2(K)$ on $\po$ is $3$-transitive, if we considered the induced action of ${\rm PGL}_2(K)$ on $n$-tuples, then the statement of Theorem \ref{tf32} would be trivially true for positive $n\leq3$ and for $n\geq 4$ it would be an easy corollary of \cite[Theorem 1]{C.1}. 

The condition on the degree is quite important too; for instance the statement of Theorem \ref{tf32} does not hold for maps of degree $4$ (see \cite[Theorem 2]{C.1}). Clearly the condition of good reduction is also crucial, as shown by Example \ref{engr}.

\bigskip

\emph{Acknowledgments.} The present work was supported by a grant from the Department of Mathematics and Computer Science at the University of Udine. The subject treated in this article was suggested to me by Pietro Corvaja. I would like to thank him for his useful advice and comments.  I am also grateful to the referee for pointing out some gaps in the previous proofs of Lemma \ref{n=34} and Lemma \ref{n=3part2} and for suggesting several improvements on the presentation of the paper.

\section{Notation and preliminary lemmas.} Throughout the present paper $K$ will be a number field, $R$ the ring of algebraic integers of $K$, $\p$ a non zero prime ideal of $R$; $\vap$ the $\p$-adic valuation on $R$ corresponding to the prime ideal $\p$ (we always assume $v_\mathfrak{p}$ to be normalized so that $v_{\mathfrak{p}}(K^*)=\mathbb{Z}$), $S$ a fixed finite set of places of $K$ including all archimedean places.\\
We shall denote by 
\bdm R_S \coloneqq\{x\in K \mid v_{\mathfrak{p}}(x)\geq0 \ \text{for every prime ideal }\ \mathfrak{p}\notin S\}\edm
the ring of $S$-integers and by 
\bdm R_S^\ast \coloneqq\{x\in K^\ast\mid v_{\mathfrak{p}}(x)=0 \ \text{for every prime ideal }\ \mathfrak{p}\notin S\}\edm
the group of $S$-units.

Let $\Phi\colon\pro\to\pro$ be a rational map defined over $K$ by $\Phi([X:Y])=[F(X,Y):G(X,Y)]$ where $F(X,Y),G(X,Y)$ are homogeneous polynomials in $R[X,Y]$ of the same degree $d$. It is easy to see that the polynomials $F$ and $G$ can be chosen with no polynomial common factors. We shall always make such a choice. 

If $H(t_1,\ldots,t_k)\in K[t_1,\ldots,t_k]$ is a non zero polynomial we define $v_{\mathfrak{p}}(H)$ as
\begin{equation}\label{vh} v_{\mathfrak{p}}(H)=v_{\mathfrak{p}}\left(\sum_{\substack{I}}a_{_I}t_1^{i_1}\cdots t_k^{i_k}\right)=\min_{\substack{I}}v_{\mathfrak{p}}(a_{_I})\end{equation}
where the minimum is taken over all multi-indices  $I=(i_i,\ldots,i_k)$. That is, $v_{\mathfrak{p}}(H)$ is the smallest valuation of the coefficients of the polynomial $H(t_1,\ldots,t_k)$. 

 We denote the discriminant of the rational map $\Phi$ by {Disc}$(\Phi)$. It is the ideal of $R$ whose valuation at a prime ideal $\mathfrak{p}$ is given by  
\begin{displaymath} v_{\mathfrak{p}}(\text{Disc}(\Phi))=v_{\mathfrak{p}}(\text{Res}(F,G))-2d\cdot\min\{v_{\mathfrak{p}}(F),v_{\mathfrak{p}}(G)\},\end{displaymath}
where $\text{Res}(F,G)$ denotes the resultant of the polynomials $F$ and $G$. By the properties of the resultant this definition is a good one, since  it does not depend on the choice of the homogeneous coefficients of the polynomials $F$ and $G$.

\begin{defin}\label{gred}We say that a morphism $\Phi\colon\mathbb{P}_1\to\mathbb{P}_1$ defined over $K$ has good reduction at a prime ideal $\mathfrak{p}$ if $v_{\mathfrak{p}}(\text{Disc}(\Phi))=0$. \end{defin}

Proposition 4.2 in \cite{M.S.2} proves that Definition \ref{gred} is equivalent to the notion of good reduction given in the introduction of this paper.

For every prime ideal $\p$, Definition \ref{gred} provides a simple condition to check whether a given rational map $\Phi$ has good reduction at a prime ideal $\p$.

Now we fix some notation useful to give the statement of some known results that we shall use in the sequel.
Let $P_1=\left[x_1:y_1\right],P_2=\left[x_2:y_2\right]$ be two distinct points in $\mathbb{P}_1(K)$ and let $\mathfrak{p}$ be a non zero prime ideal of $R$. Using the notation of  \cite{M.S.2} we shall denote by  \begin{equation}\label{d_p}\delta_{\mathfrak{p}}\,(P_1,P_2)=v_{\mathfrak{p}}\,(x_1y_2-x_2y_1)-\min\{v_{\mathfrak{p}}(x_1),v_{\mathfrak{p}}(y_1)\}-\min\{v_{\mathfrak{p}}(x_2),v_{\mathfrak{p}}(y_2)\}\end{equation}the $\mathfrak{p}$-adic logarithmic distance; $\delta_{\mathfrak{p}}\,(P_1,P_2)$ is independent of the choice of the homogeneous coordinates, i.e. it is well defined.

To every pair $P,Q\in\mathbb{P}_1(K)$ of distinct points we associate the ideal
\beq\label{I_ij} \mathfrak{I}(P,Q):=\prod_{\substack{\mathfrak{p}\notin S}}\mathfrak{p}^{\delta_{\mathfrak{p}}(P,Q)}.\eeq

To every $n$-tuple $(P_0, P_1,\ldots,P_{n-1})$ we can associate the $(n-1)$-tuple of ideals $(\mathfrak{I}_1,\mathfrak{I}_2,\ldots,\mathfrak{I}_{n-1})$ defined by  
\begin{equation}\label{I_i} \mathfrak{I}_i:=\prod_{\substack{\mathfrak{p}\notin S}}\mathfrak{p}^{\delta_{\mathfrak{p}}(P_0,P_i)}=\mathfrak{I}(P_0,P_i).\end{equation}

A cycle of length $n$ for a rational map $\Phi$ is an ordered $n$-tuple $(P_0, P_1,\ldots,P_{n-1})$ of distinct points of  $\mathbb{P}_1(K)$ with the property that $\Phi(P_i)=P_{i+1}$ for every $i\in\{0,1,\ldots,n-2\}$ and such that $\Phi(P_{n-1})=P_0$. If $\ciclon$ is a cycle for a rational map with good reduction outside $S$, then $\mathfrak{I}_i\cdot\mathfrak{I}_1^{-1}$ is an ideal of $R_S$ for every index $i\in\{1,\ldots,n-1\}$.  This last fact is an easy consequence of Proposition 5.1 and Proposition 6.1 in \cite{M.S.2}.    
Note that if two $n$-tuples are equivalent, then they have the same associated $(n-1)$-tuple of ideals.

In \cite{C.1} we proved the following:

\vspace{4mm}
\noindent{\bf \cite[Theorem 1]{C.1}} \emph{There exists a finite set $\mathbb{I}_S$ of ideals of $R_S$, depending only on $S$ and $K$, with the following property: for any cycle $(P_0, P_1,\ldots,P_{n-1})$, for a rational map of degree $\geq2$ with good reduction outside the finite set $S$, let $(\mathfrak{I}_1,\mathfrak{I}_2,\ldots,\mathfrak{I}_{n-1})$ be the associated $(n-1)$-tuple of ideals as in (\ref{I_i}); then \begin{displaymath} \mathfrak{I}_i\cdot\mathfrak{I}_1^{-1}\in\mathbb{I}_S\end{displaymath} 
for every index $i\in\{1,\ldots,n-1\}$.}

\vspace{4mm}

In the particular case of quadratic rational maps, Theorem \ref{tf32} is stronger than the above statement. Indeed, \cite[Theorem 1]{C.1} states that if  $(\mathfrak{I}_1,\mathfrak{I}_2,\ldots,\mathfrak{I}_{n-1})$ is the $(n-1)$-tuple of ideals associated to a cycle in $\po$ (for a rational map with good reduction outside $S$), then there are only finitely many possibilities for the ideal $\mathfrak{I}_i\mathfrak{I}_1^{-1}$ of $R_S$. Writing
\begin{displaymath} (\mathfrak{I}_1,\ldots,\mathfrak{I}_{n-1})=\mathfrak{I}_1\cdot(R_S,\mathfrak{I}_2\mathfrak{I}_1^{-1},\ldots,\mathfrak{I}_{n-1}\mathfrak{I}_1^{-1})\end{displaymath}
one has that for the factor $(R_S,\mathfrak{I}_2\mathfrak{I}_1^{-1},\ldots,\mathfrak{I}_{n-1}\mathfrak{I}_1^{-1})$ there are only finitely many possibilities. But, in the general case of rational maps of arbitrary degree, as shown in  \cite[Theorem 2]{C.1}, there could exist infinitely many possibilities for the factor $\mathfrak{I}_1$. Actually in \cite[Theorem 2]{C.1} we explicitly showed that the ideal $\mathfrak{I}_1$ has infinitely many possibilities, already for maps of  degree $4$.
In the present work we shall see that if one considers only quadratic rational maps, then there are only finitely many possibilities also for the factor $\mathfrak{I}_1$, when the cycle has length $\geq 3$.

Note that for cycles with length $n=2$ there exist infinitely many possibilities for $\mathfrak{I}_1$, even in the case of quadratic rational maps. Indeed, it is easy to see that every monic polynomial $\phi(x)\in R_S[x]$ induces a rational map $\Phi([X:Y])=[\phi(X/Y):1]$ which has good reduction outside $S$. For instance, for every $S$-integer $a$ the rational map $\Phi([X:Y])=[(X-Y)(X-aY):Y^2]$ has good reduction at every prime ideal $\p\notin S$, since $\Phi$ is induced by the monic polynomial $(x-1)(x-a)$, and the couple $([0:1],[a:1])$ is a cycle for $\Phi$. If two integers $a$ and $a^\prime$ generate two distinct principal ideals of $R_S$, then the cycles $([0:1],[a:1])$ and $([0:1],[a^\prime:1])$ are not equivalent modulo the action induced by ${\rm PGL}_2(R_S)$.

Since a ring of $S$-integers is not always a principal ideal domain (P.I.D.),  we shall often use the arguments contained in the following:
\begin{oss}\label{rpid}\emph{Up to enlarging the set $S$, we can suppose that $R_S$ is a P.I.D. Note that this last condition  only requires that we add    a finite number of prime ideals to $ S$. Indeed, by a simple inductive argument it follows that we have to add at most $h_{ S}-1$ prime ideals, where $h_{ S}$ is the class number of $\ri$ which is a finite number (e.g. see \cite[Chapter 5]{M.1}). Otherwise from \vvs Dirichlet's Theorem\vvd\ (e.g. see \cite[Chapter 8]{M.1}), about the uniform distribution of the prime ideals among the ideal classes, we obtain that there exist two (actually infinitely many) finite  sets   ${\ds}_1$ and ${\ds}_2$ of places of $K$ such that ${\ds}_1\cap {\ds}_2=S$ and $R_{\ds_1}$ and $R_{\ds_2}$ are principal ideal domains. }\end{oss}

The following will also be useful: By enlarging $S$ to a finite set $\ds$ so that $R_{\ds}$ is a P.I.D., we have many technical advantages. For instance we can define the \emph{greatest common divisor} of two $\ds$-integers.  Given two $\ds$-integers $a$ and $b$, every $\ds$-integer $d$ such  that $a\cdot R_{\ds}+ b\cdot R_{\ds}=d\cdot R_{\ds}$ will be called a greatest common divisor of $a$ and $b$. Clearly this definition generalizes the one on $\mathbb{Z}$. If the $\ds$-integers $a$ and $b$ are such that $a\cdot R_{\ds}+ b\cdot R_{\ds}=R_{\ds}$, we shall say that $a$ and $b$ are coprime. Furthermore, since $R_{\ds}$ is a P.I.D., the group ${\rm PGL}_2(R_{\ds})$ acts transitively on $\po$. We use this property to prove a lemma which is elementary but often useful in the sequel.

\begin{lemma}\label{Afix0} Let $R_S$ be a P.I.D.. Let $\Psi$ be a rational map defined over $K$. If there exists an automorphism $[B]\in{\rm PGL}_2(K)$ such that the rational map $[B]\circ\Psi\circ[B]^{-1}$ has good reduction outside $S$, then there exists an automorphism $[A]\in {\rm PGL}_2(K)$ of the shape $[A]=[C]\circ[B]$ for a suitable $[C]\in {\rm PGL}_2(R_S)$ such that:
\begin{enumerate} 
\item $[A]$ fixes the point $[0:1]$;
\item the rational map $[A]\circ \Psi\circ [A]^{-1}$ has good reduction outside $S$.
\end{enumerate} 
\end{lemma}
\begin{proof}The proof follows easily by taking as $[C]$ an automorphism in ${\rm PGL}_2(R_S)$ which sends the point $[B]([0:1])$ to the point $[0:1]$. Such an automorphism $[C]$ exists since the group ${\rm PGL}_2(R_S)$ acts transitively on $\po$. 
\end{proof}

 The following will also be useful:   

\begin{lemma}\label{K=S} If two quadratic rational maps, defined over $K$ with good reduction outside $S$, belong to the same orbit, with respect to the action of ${\rm PGL}_2(K)$, then they belong to the same orbit with respect to the conjugation action induced by ${\rm PGL}_2(R_S)$.\end{lemma}
\begin{proof}See Proposition 3.1 in \cite{B.2}.
\end{proof}
Notice that the statement of Lemma \ref{K=S} would be false if we considered maps of degree 1.
By Lemma \ref{K=S} we can prove Theorem \ref{tnf} and Theorem \ref{m1} considering the conjugation action induced by ${\rm PGL}_2(K)$.  Since the canonical action of ${\rm PGL}_2(K)$ on $\po$ is $3$-transitive, every quadratic rational map defined over $K$ is ${\rm PGL}_2(K)$-equivalent to a  map    of the form
\beq\label{nf} \Psi([X:Y])=[(X-\lambda Y)(aX+bY):X(aX+cY)]\eeq
for suitable $\lambda\in K$ and $a,b,c\in R_S$. Indeed, given a quadratic rational map $\Phi$ defined over $K$, we take four distinct $K$-rational points $\eta,\alpha,\beta$ and $\gamma$ such that
\bdm \eta\stackrel{\Phi}{\longmapsto}\alpha\stackrel{\Phi}{\longmapsto}\beta\stackrel{\Phi}{\longmapsto}\gamma.\edm 
Let $[A]\in {\rm PGL}_2(K)$ be the automorphism such that $[A](\alpha)=[0:1]$, $[A](\beta)=[1:0]$ and $[A](\gamma)=[1:1]$. 
Let $\lambda\in K$ such that $[A](\eta)=[\lambda:1]$, then there exist $a,b,c\in R_S$ such that $\Psi=[A]\circ\Phi\circ [A]^{-1}$ has the form as in (\ref{nf}). 
In the sequel we shall study whether a rational map of the form as in (\ref{nf}) is conjugate, via an automorphism of ${\rm PGL}_2(K)$, to a rational map with good reduction outside $S$ having $[0:1]$ as a periodic point with exact period $\geq 3$. 

As already mentioned in the introduction of this paper, an important tool for our proofs is the so called $S$-Unit Equation Theorem.

\bigskip
\noindent{\bf $S$-Unit Equation Theorem} (See \cite{E.3}\cite{VDP.1}\cite{Sc.2}). \emph{The equation
\bdm a_1x_1+a_2x_2+\ldots+a_nx_n=1,\edm
where $a_i$ are fixed elements in $K^*$ for all indices $i\in\{1,2,\ldots,n\}$, has only finitely many non-degenerate solutions
$(x_1,x_2,\ldots,x_n)\in (\ru)^n.$}

\bigskip
A solution is called non-degenerate if no subsum vanishes (i.e. $\sum_{i\in I}a_ix_i\neq 0$ for some nonempty subset $I\subsetneq\{1,2,\ldots,n\}$).
This theorem will be applied with $n$ equals to 2 and 3.
We shall also use the divisibility argument stated in Proposition 6.1 in \cite{M.S.2} and that we can write in the following form:

\vspace{4mm} 
\noindent{\bf \cite[Proposition 6.1]{M.S.2}}   \emph{If $(P_0,P_1,\ldots,P_{n-1})\subset \po$ is a cycle for a rational map, defined over $K$ with good reduction outside $S$, then
\bdm \mathfrak{I}_i=\mathfrak{I}(P_0,P_i)=\mathfrak{I}(P_k,P_{k+i})\edm
for all $i\neq0$ and $k$ in $\NN$, where $\mathfrak{I}_i$ is an ideal defined as in (\ref{I_i}).
}

If the ring $R_S$ is a P.I.D. we have a simple description of the ideals defined in  (\ref{I_ij})  . Given four $S$-integers $x_0,y_0,x_1,y_1$, if $d_0$ is a greatest common divisor of $x_0$ and $y_0$ and $d_1$ is a greatest common divisor of $x_1$ and $y_1$, then 
\bdm \mathfrak{I}([x_0:y_0],[x_1,y_1])=\left(\frac{x_0y_1-x_1y_0}{d_1d_2}\right)\cdot R_S.\edm

In order to prove our theorems we begin by proving the following lemmas.

\begin{lemma}\label{finc} Let $n\geq 4$ be a fixed positive integer. There are only finitely many $n$-tuples \beq\label{genc}([0:1],[1:0],[1:1],[\lambda_3:\mu_3],\ldots,[\lambda_{n-1}:\mu_{n-1}])\in\po\eeq that are cycles for a rational map $\Psi$, defined over $K$, for which there exists an automorphism $[A]\in{\rm PGL}_2(K)$ of $\pro$ such that the rational map $[A]\circ \Psi\circ[A]^{-1}$ has good reduction outside $S$.
\end{lemma}

\begin{proof}[Proof of Lemma \ref{finc}] Up to enlarging $S$, we can suppose that the ring $R_S$ is a principal ideal domain. Hence we choose $\lambda_i,\mu_i\in R_S$ coprime for all indices $3\leq i\leq (n-1)$.

Let $[A]\in{\rm PGL}_2(K)$ be an automorphism of $\pro$ such that the rational map $[A]\circ \Psi\circ[A]^{-1}$ has good reduction outside $S$. Since $R_S$ is a P.I.D., by Lemma \ref{Afix0}, we can suppose that the automorphism $[A]$ is induced by the matrix 
\bdm A=\begin{pmatrix}\alpha & 0\\ \beta & \gamma\end{pmatrix}\in{\rm GL}_2(K) \edm
for suitable $\alpha,\beta,\gamma\in R_S$ where $\min\{\vap(\alpha),\vap(\beta),\vap(\gamma)\}=0$. 

Clearly the $n$-tuple 
\bdm ([A]([0:1]),[A]([1:0]),[A]([1:1]),[A]([\lambda_3:\mu_3]),\ldots,[A]([\lambda_{n-1}:\mu_{n-1}]))\edm equals to 
\beq\label{Ac} ([0:1],[\alpha:\beta],[\alpha:\beta+\gamma], [\alpha\lambda_3:\beta\lambda_3+\gamma\mu_3],\ldots,[\alpha\lambda_{n-1}:\beta\lambda_{n-1}+\gamma\mu_{n-1}]),\eeq
which is a cycle for the rational map $[A]\circ \Psi\circ[A]^{-1}$. 

Let $d\in R_S$ be a greatest common divisor of $\alpha$ and $\beta$ so that the ideal $\mathfrak{I}_1$ of $R_S$ associated to the cycle (\ref{Ac}) is $(\alpha/d)\cdot R_S$.
Let $D\in R_S$ be a greatest common divisor of $\alpha$ and $\beta+\gamma$. The ideal $\mathfrak{I}_2$ of $R_S$ associated to the cycle (\ref{Ac}) is $(\alpha/D)\cdot R_S$. As seen before, $\mathfrak{I}_i\cdot\mathfrak{I}_1^{-1}$ is an ideal of $R_S$, hence $\alpha/d$ divides $\alpha/D$. Hence we deduce that $D\mid d$, therefore $D\mid \alpha$ and $D\mid \beta$. From $D\mid (\beta+\gamma)$ it follows that $D\mid\gamma$. By the assumption that $\min\{\vap(\alpha),\vap(\beta),\vap(\gamma)\}=0$ we have that $D\in R_S^*$. 
By applying \cite[Proposition  6.1]{M.S.2} to the cycle (\ref{Ac}), since $D\in R_S^*$, we see that 
\bdm \frac{\alpha}{d}\cdot R_S=\mathfrak{I}([0:1],[\alpha:\beta])=\mathfrak{I}_1=\mathfrak{I}([\alpha:\beta],[\alpha:\beta+\gamma])=\frac{\alpha\gamma}{d}\cdot R_S\edm from which  we deduce that $\gamma\in R_S^*$. Therefore by replacing $\gamma$ by $1$, $\alpha/\gamma$ by $\alpha$ and $\beta /\gamma$ by $\beta$ in (\ref{Ac}), we can rewrite the cycle (\ref{Ac}) as
\bdm ([0:1],[\alpha:\beta],[\alpha:\beta+1], [\alpha\lambda_3:\beta\lambda_3+\mu_3],\ldots,[\alpha\lambda_{n-1}:\beta\lambda_{n-1}+\mu_{n-1}]).\edm

By \cite[Theorem 1]{C.1} we can enlarge the set $S$, by adding only finitely many places, so that every ideal $\mathfrak{I}_i\cdot\mathfrak{I}_1^{-1}$, associated to a cycle for a rational map with good reduction outside $S$, is equal to $R_S$ and $R_S$ is a P.I.D.. Therefore, since $\alpha$ and $\beta+1$ are coprime, by considering the points $[0:1],[\alpha:\beta]$ and $[\alpha:\beta+1]$ we have that $\mathfrak{I}_{2}\cdot\mathfrak{I}_1^{-1}=d\cdot R_S$, thus, $d$ is a $S$-unit.

For an arbitrary index $3\leq i\leq n-1$, let $t$ be a greatest common divisor between the integers $\alpha\lambda_i$ and $\beta\lambda_i+\mu_i$. By considering the points $[\alpha:\beta]$ and $[\alpha\lambda_i,\beta\lambda_i+\mu_i]$ we deduce that $\mathfrak{I}_{i-1}\cdot\mathfrak{I}_1^{-1}=({\mu_i}/{t})\cdot R_S$, hence 
\beq\label{mu}\frac{\mu_i}{t}\in R_S^*.\eeq


By considering the points $[0:1]$ and $[\alpha\lambda_i,\beta\lambda_i+\mu_i]$ we deduce that $\mathfrak{I}_{i}\mathfrak{I}_1^{-1}=(d\lambda_i/t)\cdot R_S=(\lambda_i/t)\cdot R_S$, hence
\beq\label{lambda}\frac{\lambda_i}{t}\in R_S^*.\eeq
Since $\lambda_i$ and $\mu_i$ are coprime, by (\ref{mu}) and (\ref{lambda}), we deduce that 
\beq\label{tml}t,\lambda_i,\mu_i\in R_S^*\eeq

Now consider the points $[\alpha:\beta+1]$ and $[\alpha\lambda_i,\beta\lambda_i+\mu_i]$; since the ideal $\mathfrak{I}_{ 1  }=\alpha\cdot R_S$ (by our above assumption on $S$), it is easy to see that $\mu_i-\lambda_i\in R_S^*$. Therefore, by (\ref{tml}) and the $S$-Unit Equation Theorem, there are only finitely many possibilities for the the point $[\lambda_i:\mu_i]$.
\end{proof}
Actually the previous proof also gives the following:
\begin{lemma}\label{genA}Let $R_S$ be a P.I.D.. Let $\Psi$ be a quadratic rational map, defined over $K$, which admits a cycle as in (\ref{genc}) for suitable values of $\lambda_i,\mu_i\in R_S$ for all indices $i\in\{3,\ldots,n-1\}$. If there exists an automorphism $[B]\in {\rm PGL}_2(K)$ such that the map $[B]\circ\Psi\circ[B]^{-1}$ has good reduction outside $S$, then there exist two $S$-integers $\alpha$ and $\beta$ such that $\alpha$ and $\beta+1$ are coprime and putting 
\bdm A=\begin{pmatrix}\alpha & 0\\ \beta &1\end{pmatrix}\edm 
the rational map $[A]\circ\Psi\circ[A]^{-1}$ has good reduction outside $S$. Furthermore the two ideals ${\rm Disc}([B]), {\rm Disc}([A])$ of  $R_S$ are equal.
\end{lemma} 
The hypotheses in Lemma \ref{genA} are the same as in Lemma \ref{Afix0} with an additional condition on cycles, so we obtain a refined conclusion. 

Lemma \ref{finc} and the following Lemma \ref{fin5} will be used to prove Theorem \ref{tnf} in the case of cycles of length $\geq 5$.
\begin{lemma}\label{fin5}Let $F$ be a field. Let $n\geq 5$. Given $n-3$ fixed distinct elements $\lambda_3,\ldots,\lambda_{n-1}\in F^*$, there exists at most one endomorphism $\Psi$ of $\pro$, of degree two, defined over $F$, such that
\beq\label{c5} [0:1]\stackrel{\Psi}{\mapsto}[1:0]\stackrel{\Psi}{\mapsto}[1:1]\stackrel{\Psi}{\mapsto}[\lambda_3:1]\stackrel{\Psi}{\mapsto}[\lambda_4:1]\stackrel{\Psi}{\mapsto}\ldots\stackrel{\Psi}{\mapsto}[\lambda_{n-1}:1]\stackrel{\Psi}{\mapsto}[0:1].\eeq
\end{lemma}
\begin{proof} With every rational map $\Psi\colon \pro \to \pro$, defined over $F$, we associate in the canonical way a rational function $\psi\in F(z)$. By a change of coordinate system of $\pro$, we can assume that the cycle of $\psi$ associated to (\ref{c5}) has the first five points $z_1,\ldots,z_5$ in $F$. Therefore we have to prove that, given five distinct points $z_1,\ldots,z_5$ in $F$, if two rational functions $f, g\in F(z)$, of degree two, are such that $f(z_i)=g(z_i)$ for all indices $1\leq i \leq 5$, then they are equal. The rational function $f-g$ has degree at most $4$ and five distinct roots in $F$, thus $f\equiv g$.
\end{proof}

\section{Proof of Theorem \ref{tnf}.}

This section is dedicated to prove Theorem \ref{tnf}. Hence we are studying the case of periodic points of exact period $n\geq4$. From the previous lemmas, it is easy to see that we have to prove Theorem \ref{tnf} only in the case $n=4$. The next lemma will be useful also for the case $n=3$.

 \begin{lemma}\label{n=34}Let $\lambda, a,b,c\in R_S$, with $\min \{\vap(a),\vap(b),\vap(c)\}=0$ for every prime ideal $\p\notin S$, such that 
\begin{itemize}
\item[1)] the rational map 
\bdm\Psi([X:Y])=[(X-\lambda Y)(aX+bY):X(aX+cY)]\edm
has degree two,
\item[2)] the point $[0:1]$ is a periodic point of exact period $3$ or $4$ for $\Psi$, and
\item[3)] there exists an automorphism $[A]\in{\rm PGL}_2(K)$ such that the rational map $[A]\circ\Psi\circ[A]^{-1}$ has good reduction outside $S$.
\end{itemize}
Then 
\beq\label{png}a,b,\lambda, a\lambda+c\in \ri^*.\eeq
\end{lemma}   
\begin{proof}  Note that if $R_S$ is not a P.I.D., it is sufficient to prove Lemma \ref{n=34} for an arbitrary enlarged set $\ds$ such that $R_{\ds}$ is a P.I.D.. Indeed by Remark \ref{rpid} we may consider two enlarged sets $\ds_1$ and $\ds_2$ such that $R_{\ds_1}$ and $R_{\ds_2}$ are P.I.D. and $\ds_1\cap\ds_2=S$. If Lemma \ref{n=34} holds for $R_{\ds_1}$ and $R_{\ds_2}$ then Lemma \ref{n=34} will be true also for $S$, since $R_{\ds_1}^*\cap R_{\ds_2}^*=R_S^*$. Therefore, by enlarging the set $S$ if necessary, we can suppose that $R_S$ is a P.I.D..

By condition 1) we have that $\lambda\neq 0, a\neq 0, b\neq 0, b\neq c$ and $a\lambda\neq -c$.
If $\lambda=1$ then  $([0:1],[1:0],[1:1])$    is a cycle for $\Psi$; otherwise, by condition 2),   $([0:1],[1:0],[1:1],[\lambda:1])$    is a cycle for $\Psi$ so that

\beq\label{cc4p}[\lambda:1]=[(1-\lambda)(a+b):(a+c)].\eeq

By Lemma \ref{genA} we can suppose that the automorphism $[A]$ is induced by the matrix  
\bdm A=\begin{pmatrix}\alpha & 0\\\beta & 1\end{pmatrix}\edm  for suitable $ S$-integers $\alpha$ and $\beta$ such that $\alpha$ and $\beta+1$ are coprime $S$-integers.
Clearly the automorphism $[A]^{-1}$ is induced by the matrix
\bdm A^{-1}=\begin{pmatrix}1 & 0\\-\beta & \alpha\end{pmatrix}.\edm  We set
\beq\label{defPhi} \Phi([X:Y])\coloneqq[A]\circ \Psi\circ [A]^{-1}=[F(X,Y): G(X,Y)  ]\eeq 
where 
\begin{align}\label{f1}&F(X,Y)=\alpha((1+\beta\lambda)X-\alpha\lambda Y)((a-b\beta)X+\alpha bY),\\\label{g1}&G(X,Y)=\beta((1+\beta\lambda)X-\alpha\lambda Y)((a-b\beta)X+\alpha bY)+X((a-c\beta)X+\alpha cY).\end{align}
We rewrite the polynomials $F,G$ in a way useful for the sequel:
\begin{align}\label{f2}&\text{\small $F(X,Y)=\alpha(1+\beta\lambda)(a-b\beta)X^2+\alpha^2(2b\beta\lambda+b-a\lambda)XY-\alpha^3b\lambda Y^2$},\\\label{g2}&\text{\small $G(X,Y)=[\beta(1+\beta\lambda)(a-b\beta)+a-c\beta]X^2+\alpha(2\beta^2b\lambda +\beta b-\beta a\lambda+c)XY-\alpha^2 b\beta\lambda Y^2$}.\end{align}
Note that the $ S$-integer $\alpha^2$ divides both the coefficients of the monomial in $Y^2$ of $F$ and $G$. Since $\Phi$ has good reduction outside $ S$ then $\alpha^2$ divides all the coefficients of the polynomials $F$ and $G$ as represented in (\ref{f2}) and (\ref{g2}). Hence we can choose
\begin{align}\label{f3}&\text{\small $F(X,Y)= \frac{(1+\beta\lambda)(a-b\beta)}{\alpha}X^2+(2b\beta\lambda+b-a\lambda)XY-\alpha b\lambda Y^2$},\\\label{g3}&\text{\small $G(X,Y)=\frac{[\beta(1+\beta\lambda)(a-b\beta)+a-c\beta]}{\alpha^2}X^2+\frac{(2\beta^2 b\lambda+\beta b-\beta a\lambda+c)}{\alpha}XY-b\beta\lambda Y^2$},\end{align} where the coefficients of all monomials in (\ref{f3}) and (\ref{g3}) are in $\ri$ and (\ref{defPhi}) still holds.

Now it is easy to prove that $b\in\ri^*$. Indeed, from  the argument above concerning $\alpha$, it follows that the $ S$-integer $b$ divides the coefficients of the polynomials $F$ and $G$ as represented in (\ref{f3}) and (\ref{g3}). By considering the coefficient of the monomial in $XY$ in (\ref{f3}) we see that $b|a$ and from the coefficient of $XY$ in (\ref{g3}) we have $b|c$. By the assumption that $\min \{\vap(a),\vap(b),\vap(c)\}=0$ for every prime ideal $\p\notin S$, we deduce that $b\in\ri^*$. To ease notation we replace $a$ with $a/b$, $b$ with 1 and $c$ with $c/b$ in (\ref{f3}) and (\ref{g3}) obtaining that
\begin{align}\label{f4}&\text{\small $F(X,Y)=\dfrac{(1+\beta\lambda)(a-\beta)}{\alpha}X^2+(2\beta\lambda+1-a\lambda)XY- \alpha\lambda Y^2$},\\\label{g4}&\text{\small $G(X,Y)=\dfrac{[\beta(1+\beta\lambda)(a-\beta)+a-c\beta]}{\alpha^2}X^2+\dfrac{(2\beta^2\lambda +\beta -\beta a\lambda+c)}{\alpha}XY-\beta\lambda Y^2$}\end{align}
where, clearly, the equality (\ref{defPhi}) still holds and the polynomials in (\ref{f4}) and (\ref{g4}) still have coefficients in  $\ri$. Therefore we have that
\begin{align}\label{N2}&\dfrac{(1+\beta\lambda)(a-\beta)}{\alpha}\in\ri,\\
\label{D2}&\dfrac{\beta(1+\beta\lambda)(a-\beta)+a-c\beta}{\alpha^2}\in\ri,\\
\label{D1}&\dfrac{2\beta^2\lambda +\beta -\beta a\lambda+c}{\alpha}\in\ri.\end{align}
 Moreover, again by good reduction of $\Phi$ outside $S$, we see that $\lambda$ is a $S$-unit. Indeed, note that $\lambda$ divides the coefficient of the monomial $Y^2$ in both (\ref{f4}) and (\ref{g4}). Hence $\lambda$ has to divide $2\beta\lambda+1-a\lambda$ (which is the coefficient of monomial $XY$ in (\ref{f4})), so $\lambda$ is a $S$-unit.  

The next step is to prove that $a\in\ri^*$. First we shall see that $\alpha$ and $1+\beta\lambda$ are coprime. After that we shall prove that $a$ and $\alpha$ are coprime and that $a\in\ri^*$. 

 Let $\lambda=1$. By the assumption on $A$ (see Lemma \ref{genA}) we have that the $S$-integers $\alpha$ and $1+\beta\lambda=1+\beta$  are coprime. Let $\lambda\neq1$, i.e. we are in the case of cycle of length $4$. Equation (\ref{cc4p}) gives
\beq\label{cc4}(1-2\lambda)a-\lambda c=\lambda-1.\eeq  
Suppose that there exists a non zero prime ideal $\p$ which divides $\alpha$ and $1+\beta\lambda$ so that 
\beq\label{ap}\beta\lambda\equiv-1 \ \ \ \text{(mod $\p$)}.\eeq 
By (\ref{N2}) and (\ref{D2}) $\p$ divides $a-c\beta$, thus
\beq\label{acp}  a\lambda+c\equiv 0\ \ \ \text{(mod $\p$)}.\eeq
By (\ref{D1}) and (\ref{ap})
\beq\label{betac}-\beta+a+c\equiv 0 \ \ \ \text{(mod $\p$)}.\eeq 
Now, by (\ref{acp}) and (\ref{betac})
\bdm 0\equiv-\beta+a+c\equiv-\beta\lambda+a\lambda +c\lambda\equiv 1+a\lambda-a\lambda^2 \ \ \ \text{(mod $\p$)}\edm
thus
\beq\label{p1}(\lambda^2-\lambda)a\equiv 1\ \ \ \text{(mod $\p$)}.\eeq 
By (\ref{acp}), the equality (\ref{cc4}) gives
\bdm (\lambda-1)^2a\equiv\lambda-1\ \ \ \text{(mod $\p$)}.\edm
From this by (\ref{p1}) we deduce that
\bdm \lambda(\lambda-1)\equiv\lambda(\lambda-1)^2a\equiv\lambda-1\ \ \ \text{(mod $\p$)} ,  \edm
which is equivalent to 
\bdm (\lambda-1)^2\equiv 0\ \ \ \text{(mod $\p$)}.\edm
But this last congruence contradicts the one in (\ref{p1}). We have proved that 
\beq\label{ccop} \text{\emph{$\alpha$ and $1+\beta\lambda$ are coprime $S$-integers}}.\eeq
 From (\ref{N2}) we deduce 
\beq\label{NN2} \dfrac{a-\beta}{\alpha}\in\ri,\eeq
 and by (\ref{D2}) we obtain that $(a-c\beta)/\alpha\in \ri$; therefore by  
(\ref{f4}) and (\ref{g4}) it follows that  
\begin{align}\label{f5}&F(X,Y)=\left(\left(1+\beta\lambda\right)X-\alpha\lambda  Y\right)\left(\dfrac{a-\beta}{\alpha}X+Y\right)\\\label{g5}&G(X,Y)=\dfrac{1}{\alpha}\left[\beta\left(\left(1+\beta\lambda\right)X-\alpha \lambda Y\right)\left(\dfrac{a-\beta}{\alpha}X+Y\right)+X\left(\dfrac{a-c\beta}{\alpha}X+ cY\right)\right],\end{align}
 where $\left(\dfrac{a-\beta}{\alpha}X+Y\right)$ and $\left(\dfrac{a-c\beta}{\alpha}X+ cY\right)$ are polynomials with coefficients in $R_S$.

Let $\bar{\alpha}$ be a greatest common divisor of  $a$ and $\alpha$.
By (\ref{NN2}) the $S$-integer $\bar{\alpha}$ divides $\beta$; since $\bar{\alpha}$ divides the coefficients of $Y^2$ in (\ref{f4})  and   (\ref{g4}) and the map $\Phi$ has good reduction, $\bar{\alpha}$ has to divide all the coefficients of $F$ and $G$ as represented in (\ref{f4}) and (\ref{g4}). Thus 
\bdm \bar{\alpha}\mid 2\beta\lambda +1-a\lambda\ \ \Rightarrow\ \ \bar{\alpha}\mid 1\ \ \Rightarrow\ \ \ \bar{\alpha}\in\ri^*.\edm
This argument also proves that $\alpha$ and $\beta$ are coprime $S$-integers. Indeed by (\ref{NN2}) every prime which divides $\alpha$ and $\beta$ also divides $a$.

  Since the polynomials $\left(\dfrac{a-\beta}{\alpha}X+Y\right)$ and $\left(\dfrac{a-c\beta}{\alpha}X+ cY\right)$ belong to $R_S[X,Y]$, from  the representation of polynomials  $F(X,Y)$ and $G(X,Y)$ in (\ref{f5}) and (\ref{g5}) and the good reduction of the rational map $\Phi$ outside $S$, we deduce that every prime ideal $\p\notin S$ which divides  
\beq\label{aa+c} \left|\begin{matrix}\dfrac{a-\beta}{\alpha} & 1\\\dfrac{a-c\beta}{\alpha} & c\end{matrix}\right|=\dfrac{a(c-1)}{\alpha}\in\ri\eeq
has to divide all the coefficients of the polynomials $F$ and $G$ as represented in (\ref{f4}) and (\ref{g4}).  Let $\p\notin S$ be a prime ideal which divides $a$. Thus, since $\alpha$ and $a$ are coprime, $\p$ must divide all the coefficients of the polynomials $F$ and $G$ as represented in (\ref{f4}) and (\ref{g4}), in particular $\p|\alpha\lambda$. But, since $\lambda\in R_S^*$, it contradicts the coprimality of $\alpha$ and $a$. Hence  $a\in\ri^*$. Note that (\ref{aa+c}) also tells us that 
\beq\label{ac1}\alpha\mid(c-1).\eeq

Now, by the same above argument applied to  the polynomials $(1+\beta\lambda)X-\alpha\lambda Y$ and $\dfrac{a-c\beta}{\alpha}X+ cY$ which compare in the representation of polynomials $F$ and $G$ in (\ref{f5}) and (\ref{g5}), we see that 
\beq\label{a+c} \left|\begin{matrix}1+\beta\lambda & -\alpha\lambda\\\dfrac{a-c\beta}{\alpha} & c\end{matrix}\right|=a\lambda+c\ \ \text{divides $\alpha\lambda$}.\eeq 
Since $\lambda\in R_S^*$, (\ref{a+c}) tells us that $a\lambda+c$ divides $\alpha$. Thus, by (\ref{NN2}), $a\lambda+c$ also divides $a-\beta$ and by (\ref{ac1}) $a\lambda+c\mid c-1$, so
\bdm a\lambda+c\mid a\lambda+c-c+1=a\lambda+1\ \ \Rightarrow\ \ a\lambda+c\mid a\lambda+1-a\lambda+\beta\lambda=1+\beta\lambda\edm
therefore, by (\ref{a+c}) and (\ref{ccop}), it follows that $a\lambda+c\in\ri^*$.
\end{proof}

The next lemma concerns only cycles of length $n=4$.

\begin{lemma}\label{fin4} There exist only finitely many quadratic rational maps $\Psi\colon \pro\to \pro$, defined over $K$, such that 
\begin{itemize}
\item[i)] there exists a point $[\lambda:\mu]\in\po$ such that 
\beq\label{cy4} [0:1]\stackrel{\Psi}{\mapsto}[1:0]\stackrel{\Psi}{\mapsto}[1:1]\stackrel{\Psi}{\mapsto}[\lambda:\mu]\stackrel{\Psi}{\mapsto}[0:1];\eeq
\item[ii)] there exists an automorphism $[A]\in {\rm PGL}_2(K)$ of $\pro$ such that the rational map $[A]\circ\Psi\circ[A]^{-1}$ has good reduction outside $S$.
\end{itemize}

\end{lemma}
\begin{proof}As already remarked, without less of generality we suppose that $S$ is enlarged so that $R_S$ is a principal ideal domain. Clearly if $[\lambda:\mu]\in\{[1:0],[0:1],[1:1]\}$, then (\ref{cy4}) does not hold for any endomorphism $\Psi$. So that we can suppose that $\lambda\mu\neq 0$ and $\lambda\neq\mu$. By Lemma \ref{finc} there are only finitely many possibilities for $[\lambda:\mu]\in\po$. Hence it is sufficient to prove the lemma for such a  given point   $[\lambda:\mu]$. Let $(\lambda,\mu)\in R_S^2$ be homogeneous coprime coordinates representing the point $[\lambda:\mu]$. Thus, by adding to $S$ all the finitely many prime ideals which contain $\mu$ and $\lambda$, we can suppose that $\mu$ and $\lambda$ are $S$-units. By dividing by $\mu$, we have that $\mu=1$ and $\lambda\neq 1$.

Since $\Psi$ has to satisfy condition (\ref{cy4}), there exist three $S$-integers $a,b,c$ without common divisors (i. e. $\min \{\vap(a),\vap(b),\vap(c)\}=0$ for every $\p\notin S$) such that
\bdm \Psi([X:Y])=[(X-\lambda Y)(aX+bY):X(aX+cY)]\edm where $a\neq 0, b\neq 0, b\neq c$ and $a\lambda\neq -c$. Therefore we are in the hypotheses of Lemma \ref{n=34} so (\ref{png}) holds. Since $\lambda\neq 1$, (\ref{cc4p}) also holds, hence it follows that
\beq\label{cc4n} (1-2\lambda)a-\lambda c=b(\lambda-1).\eeq


We set $a\lambda+c=v$; by (\ref{png}) the integers $a,b,v$ are $S$-units. Therefore (\ref{cc4n}) gives $(1-2\lambda)a-\lambda v+\lambda^2 a=b(\lambda-1)$ which is equivalent to the following $S$-unit equation 
\bdm (\lambda-1)a-\frac{\lambda}{\lambda-1} v=b\edm
(recall that $\lambda\neq 1$ is fixed). By $S$-Unit Equation Theorem there are only finitely many solutions $(a/b,v/b)\in R_S^*\times R_S^*$. Clearly it follows that there are only finitely many possibilities for $[a:b:c]\in \mathbb{P}_2(K)$. This concludes the proof.
\end{proof}

\begin{proof}[Proof of Theorem \ref{tnf}]
By Lemma \ref{K=S} we can consider the action of ${\rm PGL}_2(K)$ instead of ${\rm PGL}_2(R_S)$. Therefore we are studying classes of quadratic rational maps which contain a rational map $\Psi$ of the form as in (\ref{nf}), which admits $[0:1]$ as a periodic point of exact period $n$, and an automorphism $[A]\in {\rm PGL}_2(K)$ such that the map $[A]\circ\Psi\circ[A]^{-1}$ has good reduction outside $S$. If $n\geq 5$ then Lemma \ref{finc} and Lemma \ref{fin5} give the proof of Theorem \ref{tnf}. For the case $n=4$ we are done, by Lemma \ref{fin4}.
\end{proof}

As already written in the introduction of this paper, to prove Theorem 1' it is sufficient now to apply the bound proved by Morton and Silverman in \cite{M.S.1}.

\section{Periodic points with exact period equal to 3.}

As already mentioned in the introduction, for the proof of Theorem \ref{m1} (the case $n=3$), we need some results proved by P. Corvaja and U. Zannier in \cite{C.Z.1} which were obtained via the so called Subspace Theorem proved by W.M. Schmidt (e.g. see \cite{Sc.1} or \cite{Sc.2}).
Actually the Corvaja and Zannier's results are used to prove a technical lemma which only involves divisibility arguments. Hence we have dedicated a subsection to this preliminary result. This technical lemma will be used in the second part of this section to prove the following proposition which represents an important tool to prove Theorem \ref{m1} and Theorem \ref{tf32}:

\begin{prop}\label{n=3} There exists a finite set $\mathcal{B}$, depending only on $S$, such that if $\Psi$ is a quadratic rational map defined over $K$ satisfying 
\beq\label{pp0} [0:1]\stackrel{\Psi}{\mapsto}[1:0]\stackrel{\Psi}{\mapsto}[1:1]\stackrel{\Psi}{\mapsto}[0:1]\eeq
and equivalent, with respect to the action of ${\rm PGL}_2(K)$, to a rational map with good reduction outside $S$, then there exist two $S$-integers $a$ and $c$ such that the map $\Psi$ has the form 
\beq\label{nPsi3} \Psi(X,Y)=[(X-Y)(aX+Y):X(aX+cY)]\eeq
and
\begin{itemize}\item[1)]$a\in \ru$ and $a+c\in \ru$;
\item[2)]$a,c\in\mathcal{B}$ or one of the following conditions holds:
\begin{itemize}\item[i)] $a=-1$,
\item[ii)]$c=0$,
\item[iii)]$c=1-a$.
\end{itemize}\end{itemize}
Furthermore if an automorphism $[A]$ is such that the rational map $[A]\circ\Psi\circ[A]^{-1}$ has good reduction outside $S$, then the ideal ${\rm Disc}([A])$ of $\ri$ has only finitely many possibilities which only depend on $S$.
\end{prop}

\subsection{A technical lemma.}

We set the following notation:
\bdm \log^+x\coloneqq\max\{0,\log x\}\ \ ,\ \ \log^-x\coloneqq\min\{0,\log x\},\ \ x>0.\edm
Let $M_K$ denote the set of places of $K$. For a place $\mu\in M_K$, we shall denote by $|\cdot|_{\mu}$ the corresponding absolute value normalized with respect to $K$ so that the product formula 
\bdm \prod_{\mu\in M_K}|x|_{\mu}=1\edm
holds for all $x\in K\setminus\{0\}$. For any element $x\in K$, let $H(x)$ denote the absolute multiplicative height of the point $[1:x]\in\po$
\bdm H(x)=\prod_{\mu\in M_K}\max \{1,|x|_{\mu}\}.\edm
We shall denote by $h(x)$ the logarithmic height $h(x)=\log H(x)$.

Below we state three of the results, proved by P. Corvaja and U. Zannier in \cite{C.Z.1}, in a simpler form and adapted to our situation.

\vspace{4mm}
\noindent{\bf \cite[Main Theorem]{C.Z.1}}. \emph{Let $f(X,Y)\in\overline{\mathbb{Q}}(X,Y)$ be a rational function. Suppose that $1$ appears as a monomial in the numerator or in the denominator of $f(X,Y)$. Then for every $\epsilon>0$ the Zariski closure of the set of solutions $(u,v)\in (R_S^*)^2$ of the inequality 
\bdm h(f(u,v))<(1-\epsilon)\max \left\{\dfrac{h(u)}{2\deg_Y f},\dfrac{h(v)}{2\deg_X f}\right\}\edm
is a finite union of translates of proper subtori of ${\rm \bf G}_m^2.$}

\vspace{4mm}
\noindent{\bf \cite[Proposition 1]{C.Z.1}}. \emph{Let $f(X,Y)\in K[X,Y]$ be a polynomial with $f(0,0)\neq 0$. Then for every $\epsilon>0$, all but finitely many solutions $(u,v)\in (R_S^*)^2$ to the inequality 
\bdm \sum_{\mu\in S}\log^-|f(u,v)|_{\mu}<-\epsilon\max \left\{h(u),h(v)\right\}\edm
lie in the union of finitely many translates of $1$-dimensional subtori of ${\rm \bf G}_m^2$, which can be effectively determined.}

\vspace{4mm}
\noindent{\bf \cite[Proposition 4]{C.Z.1}}. \emph{Let $r(X),s(X)\in\overline{\mathbb{Q}}[X]$ be two non zero polynomials. Then for every $\epsilon>0$, all but finitely many solutions $(u,v)\in (R_S^*)^2$ to the inequality 
\bdm \sum_{\mu\in M_K\setminus S}\log^-\max\left\{|r(u)|_{\mu},|s(v)|_{\mu} \right\}  <-\epsilon\max \left\{h(u),h(v)\right\}\edm
lie in the union of finitely many translates of $1$-dimensional subtori of ${\rm \bf G}_m^2$, which can be effectively determined.}

\vspace{4mm}

These results are used to prove the following technical lemma:
\begin{lemma}\label{n=3part2}There exists a finite set $T$ of prime ideals of $R$, containing $S$, with the following property: if $\alpha$ is a non zero $S$-integer for which there exist three $S$-units $a$, $u$ and $v$ such that 
\beq\label{alpha2}u\alpha^2=v-a-1\eeq

\beq\label{da} \alpha\mid (1+a+a^2)\ \ \text{in $R_{S}$}\eeq
and
\beq\label{dv} \alpha\mid (1-v+v^2)\ \ \text{in $R_{S}$},\eeq then $\alpha$ is a $T$-unit.
\end{lemma}

\begin{proof}
If the right side term in equation (\ref{alpha2}) has a vanishing subsum, then it is sufficient to take $T=S$. Therefore we suppose that the right side term in equation (\ref{alpha2}) has no vanishing subsums. In order to use \cite[Main Theorem]{C.Z.1}, \cite[Proposition 1]{C.Z.1} and \cite[Proposition 4]{C.Z.1}, note that under (\ref{alpha2}), (\ref{da}), (\ref{dv}), $v-a-1\mid (1+a+a^2)^2$ and $v-a-1\mid (1-v+v^2)^2$ in $R_{S}$.
To ease notation we set 
\bdm f(X,Y)=Y-X-1\ \ ,\ \ r(X)=(1+X+X^2)^2\ \ ,\ \ s(X)=(1-X+X^2)^2,\edm
so $f(a,v)$ divides $r(a)$ and $s(v)$ in $R_S$.
Note that if $x,y\in\ri$ then $\log^-\left(|x|_{\mu}\cdot|y|_{\mu}\right)=\log^-|x|_{\mu}+\log^-|y|_{\mu}$ for all $\mu\in M_K\setminus S$. Hence, since $r(a)/f(a,v)$ and $s(v)/f(a,v)$ belong to $R_{S}$, we have that
\begin{multline*} -\sum_{\mu\in M_{ K}\setminus S}\log^-\max\{|r(a)|_{\mu},|s(v)|_{\mu}\}=\\-\sum_{\mu\in M_{ K}\setminus S}\log^-|f(a,v)|_{\mu}-\sum_{\mu\in M_{ K}\setminus S}\log^-\max\left\{\left|\dfrac{r(a)}{f(a,v)}\right|_{\mu},\left|\dfrac{s(v)}{f(a,v)}\right|_{\mu}\right\}\end{multline*}
thus
\beq\label{acz1}-\sum_{\mu\in M_{ K}\setminus S}\log^-|f(a,v)|_{\mu}\leq -\sum_{\mu\in M_{ K}\setminus S}\log^-\max\{|r(a)|_{\mu},|s(v)|_{\mu}\}.\eeq
Now by the product formula we have that $\sum_{\mu\in M_{ K}}\log|f(a,v)|_{\mu}=0$ and by definition of the logarithmic height it follows that $h(f(a,v))=\sum_{\mu\in M_{ K}}\log^+|f(a,v)|_{\mu}$. Hence
\bdm -\sum_{\mu\in M_{ K}\setminus S}\log^-|f(a,v)|_{\mu}=h(f(a,v))+\sum_{\mu\in S}\log^-|f(a,v)|_{\mu}.\edm
  Since every point of $(R_S^*)^2$ is contained in a suitable translate of a $1$-dimensional subtorus of ${\rm \bf G}_m^2$, by applying \cite[Proposition 1]{C.Z.1} with $\epsilon=1/4$ we obtain from the above identity that for all $(a,v)\in(\ri^*)^2$ outside a finite union of finitely many  translates    of $1$-dimensional subtori of ${\rm \bf G}_m^2$ 
\bdm -\sum_{\mu\in M_{ K}\setminus S}\log^-|f(a,v)|_{\mu}\geq h(f(a,v))-\dfrac{1}{4}\max\{|h(a),h(v)\}.\edm
By applying \cite[Main Theorem]{C.Z.1} with $\epsilon=1/4$ to the last inequality we have 
\beq\label{acz2} -\sum_{\mu\in M_{ K}\setminus S}\log^-|f(a,v)|_{\mu}\geq \frac{3}{8}\max\{|h(a),h(v)\}-\dfrac{1}{4}\max\{h(a),h(v)\}=\dfrac{1}{8}\max\{h(a),h(v)\}\eeq
holds for all $(a,v)\in(\ri^*)^2$ outside a finite union of finitely many translates of $1$-dimensional subtori of ${\rm \bf G}_m^2$.

Now applying \cite[Proposition 4]{C.Z.1} with $\epsilon=1/9$ we obtain that for all $(a,v)\in(\ri^*)^2$ outside a finite union of finitely many translates of $1$-dimensional subtori of ${\rm \bf G}_m^2$,
\beq\label{acz3} 
-\sum_{\mu\in M_{ K}\setminus S}\log^-\max\{|r(a)|_{\mu},|s(v)|_{\mu}\}\leq \frac{1}{9}\max\{h(a),h(v)\}.\eeq
Putting together (\ref{acz1}), (\ref{acz2}) and (\ref{acz3}) we obtain that for all $(a,v)\in(\ri^*)^2$ outside a finite union of finitely many translates of $1$-dimensional subtori of ${\rm \bf G}_m^2$
\bdm \frac{1}{9}\max\{|h(a),h(v)\}\geq \frac{1}{8}\max\{h(a),h(v)\}\edm
holds, which means $\max\{h(a),h(v)\}=0$.
 Since there are only finitely many algebraic numbers of bounded degree and bounded height (see for example Theorem 1.6.8 in \cite{Bo.1} or Theorem B.2.3 in \cite{H.S.1})  , there are only finitely many $(a,v)\in(R_{S}^*)^2$ which satisfy $\max\{h(a),h(v)\}=0$. \\
 Therefore, by using again that every point of $(R_S^*)^2$ is contained in a suitable translate of a $1$-dimensional subtorus of ${\rm \bf G}_m^2$, we can summarize all the above arguments by saying that there exists a finite union of translates of $1$-dimensional subtori of ${\rm \bf G}_m^2$ which contains every pair $(a,v)$ satisfying the hypotheses of Lemma  \ref{n=3part2}.   

Recall that every translate of a proper subtorus of ${\rm \bf G}_m^2$ is either a point or a curve defined by an equation of the form 
\bdm A^pV^q=\omega\edm
for coprime $p,q\in\ZZ$ and nonzero $\omega\in K^*$. By the above arguments, there exist only finitely many coprime $p,q\in\ZZ$ and $\omega\in K^*$ such that for every couple $(a,v)\in(R_{S}^*)^2$ satisfying  the hypotheses of Lemma  \ref{n=3part2}, one of the following finitely many identities
\beq\label{omegaeq} a^pv^q=\omega\eeq 
holds. 
Without loss of generality we can suppose $p>0$. By Dirichlet Unit Theorem (e.g. see \cite[pag. 273]{H.S.1}) the group $R_{S}^*$ is finitely generated. Thus there exists a finite extension $F$ of $K$ which contains all $p$-th roots  of all $S$-units. Since $a^p=\omega v^{-q}$ we can write the $ S$-unit $a$ in the following way
\beq\label{ale} a=\lambda(\eta)^{-q}\eeq 
where $\lambda,\eta\in F$ are one of the  $p$-th roots    of $\omega$ and $v$ respectively. Of course we have to enlarge $S$ by taking all places of the extension $F$ which extend all places in $S$. We shall denote this set by $\mathbb{S}$. Moreover we add finitely many prime ideals to $\mathbb{S}$ so that the ring of $\mathbb{S}$-integers of $F$ is a principal ideal domain. 
The equality (\ref{alpha2}) becomes
\beq\label{alp} u\alpha^2=\eta^p-\lambda(\eta)^{-q}-1.\eeq
If $q<0$, then $u\alpha^2=f(\eta)$ where 
\beq\label{ft1}f(t)=t^p-\lambda t^{-q}-1\in F[t].\eeq 
If $q>0$ then 
\beq\label{alp1}u\eta^q\alpha^2=\eta^{p+q}-\eta^q-\lambda\eeq
so $u\eta^q\alpha^2=f(\eta)$ where 
\beq\label{ft2}f(t)=t^{p+q}-t^q-\lambda\in F[t].\eeq 
Since there are only finitely many possibilities for $p,q\in\ZZ$ and $\omega\in K$, there are only finitely many possibilities for $f(t)$.

Now we prove that there exists a finite set $T$ such that $\alpha\in T$-units.  Let $\rho_1$ and $\rho_2$ be the primitive cube roots of unity and $\zeta_1$ and $\zeta_2$ be the primitive sixth roots of unity. We enlarge $F$ and $\mathbb{S}$ so that the ring $Q_{\mathbb{S}}$ of $\mathbb{S}$-integers of $F$ is a P.I.D. containing all $6$-th roots of unity $\rho_1$, $\rho_2$, $\zeta_1$ and $\zeta_2$. 
Let $\p$ be a prime ideal of $Q_{\mathbb{S}}$ which divides $\alpha$.  By (\ref{da}) and (\ref{ale})
\bdm 0\equiv1+\lambda\eta^{-q}+\lambda^2\eta^{-2q}\equiv (\lambda\eta^{-q}-\rho_1)(\lambda\eta^{-q}-\rho_2)\ \ \ \text{(mod$\p$)}\edm
hence
\beq\label{et} \lambda\eta^{-q}\equiv\rho_i\ \ \ \text{(mod$\p$)}\ \ \text{with $i=1$ or $i=2$}.\eeq
By (\ref{dv})
\bdm 0\equiv\eta^{2p}-\eta^p+1\equiv(\eta^p-\zeta_1)(\eta^p-\zeta_2)\ \ \ \text{(mod$\p$)}\edm
hence
\beq\label{ep} \eta^p\equiv\zeta_j\ \ \text{(mod$\p$)}\ \ \text{with $j=1$ or $j=2$}.\eeq
Note that $\eta\nequiv 0 \ \text{(mod$\p$)}$. Now (\ref{et}) and (\ref{ep}) give 
\beq\label{etp2}\lambda^p\eta^{-pq}\equiv\rho_i^p\ \ ,\ \ \lambda^p\eta^{-pq}\equiv \lambda^p\zeta_j^{-q}\ \ \ \text{(mod$\p$)}\eeq
so 
\bdm \lambda^p\equiv\rho_i^p\zeta_j^{q}\ \ \ \text{(mod$\p$)},\edm
i.e. the prime ideal $\p$ divides $\lambda^p-\rho^{p}_i\zeta^{q}_j$. We have proved that there exist only finitely many possibilities for $\p$ unless $\omega=\lambda^p=\zeta_j^{q}\rho_i^{p}$. Of course this last equality implies that $\lambda$ is a root of unity. We denote by $T_1$ the set of prime ideals defined by
\small\bdm T_1\coloneqq\left\{\p\ \text{prime ideal}\left| \begin{array}{l}\text{$\p$ divides a non zero $\mathbb{S}$-integer of the form $\omega-\zeta_j^{q}\rho_i^{p}$,  }\\\text{where $(\omega,p,q)\in R_S^*\times\mathbb{Z}\times\mathbb{Z}$ is associated }\\\text{to one of the finitely many identities (\ref{omegaeq})}\end{array}\right.\right\}.\edm\normalsize
Clearly $T_1$ is a finite set.
Therefore if any $\omega$, which satisfies (\ref{omegaeq}), is not a root of unity we are done. Otherwise, if $\lambda$ is a root of unity, then we prove that any polynomial $f(t)$ like (\ref{ft1}) and (\ref{ft2}) has a simple root. 
Actually we prove the simple and more general result:

\begin{lemma}\label{lsr}Let $f(t)\in\mathbb{C}[t]\setminus\mathbb{C}$ be a polynomial of the form $t^m+\lambda_1t^n+\lambda_2$ with $|\lambda_1|=|\lambda_2|=1$, where $|\cdot|$ denotes the usual absolute value on $\mathbb{C}$. Then $f(t)$ has at least one simple root.\end{lemma}
\begin{proof}If $m=n$ the lemma is trivial. Let $m\neq n$, without loss of generality we can suppose that $m>n$.  If $\zeta$ is a repeated root of the polynomial $f(t)$, then $\zeta$ is a root of $f^\prime(t)=mt^{m-1}+\lambda_1nt^{n-1}$. Thus, since $\zeta\neq 0$, $\zeta^{m-n}=-\lambda_1n/m$. If $f(t)$ had only repeated roots, then 
\bdm |\lambda_2|=\left|-\lambda_1\frac{n}{m}\right|^{\frac{m}{m-n}}\edm
which is trivially false since $|\lambda_2|=|\lambda_1|=1$, $m/(m-n)\neq 0$ and $|n/m|\neq 1$.
\end{proof}
Now we prove a lemma which is a simple application of Siegel Theorem about the finiteness of $\mathbb{S}$-integer points of a curve of genus $\geq 1$ (e.g. see \cite[Chapter 7]{S.1}).

\begin{lemma}\label{lfu}Let $f(t)\in F[t]$ be a polynomial which has at least one simple root and does not vanish at zero. Then, the equation 
\beq\label{l2}y^2=f(u)\eeq
has only finitely many solution $(y,u)\in Q_{\mathbb{S}}\times Q_{\mathbb{S}}^*$.
\end{lemma}
\begin{proof}
Let $\mathcal{C}$ be the affine curve defined by the equation $y^2=f(t)$. For any double root $a$ of the polynomial $f$, we can replace $y$ by $(t-a)y$ and cancel $(t-a)^2$. In this way we can assume that $\mathcal{C}$ is given by the equation $y^2=g(t)$ where $g(t)\in F[t]\setminus F$ has no repeated roots; thus $\mathcal{C}$ is an affine smooth curve. Recall that the group of $\mathbb{S}$-units $Q_{\mathbb{S}}^*$ is a finitely generated group, hence there exists a finite set $U$ such that every $\mathbb{S}$-unit $u$ is such that $u=v\cdot w^3$, for suitable $v\in U$ and $w\in Q_{\mathbb{S}}^*$. Therefore if $(y,u)\in Q_{\mathbb{S}}\times Q_{\mathbb{S}}^*$ is a solution of (\ref{l2}), then
\beq\label{G} y^2=g(u)=G(w)\eeq where $G\in F[t]$ is a polynomial with degree $3n\geq 3$ without repeated roots so that the curve defined by the affine equation $y^2=G(t)$ has genus $\geq 1$. By Siegel Theorem there are only finitely many solutions $(y,w)\in Q_{\mathbb{S}}\times Q_{\mathbb{S}}$ to the equation (\ref{G}) so that the lemma is proved.

\end{proof} 

By (\ref{alp}) and (\ref{alp1}) there exists an $\mathbb{S}$-unit $w$ (equal to $u$ or  $u\eta^q$  ) such that $w\alpha^2=f(\eta)$, where $f(t)$ is one of the finitely many polynomials described in (\ref{ft1}) and (\ref{ft2}). If $\lambda$ is a root of unity, then any such polynomial $f(t)$ satisfies the hypothesis of Lemma \ref{lsr}. Indeed any polynomials $f(t)$ as in (\ref{ft1}) and (\ref{ft2}) has degree $\geq 1$ since we have supposed that the right side term in equation (\ref{alpha2}) has no vanishing subsums. Notice that there exists a finite set $V$ of $\mathbb{S}$-units such that for every $\mathbb{S}$-unit $w$ we have  $w=s\cdot l^2$ for suitable $l\in Q_{\mathbb{S}}^*$ and $s\in V$. So by Lemma \ref{lfu} there are only finitely many possibilities for $u\alpha^2$ and  $u\eta^q\alpha^2$.

We denote by $T_2$ the following set of prime ideals of $Q_{\mathbb{S}}$: 
\bdm T_2=\left\{\p\ \text{prime ideal}\left| \begin{array}{l}\text{$\p$ divides an $\mathbb{S}$-integer $\alpha$ such that there exist $\mathbb{S}$-units }\\\text{$a$, $v$ and a root of unity $\omega$ satisfying the hypotheses }\\\text{of Lemma \ref{n=3part2} and one of the finitely many identities (\ref{omegaeq})}\end{array}\right.\right\}.\edm
The above arguments prove that $T_2$ is a finite set. Now it is sufficient to take 
$T=\left\{ \p\cap K\mid \p\in \mathbb{S}\cup T_1\cup T_2\right\}.$
\end{proof}

\subsection{Proof of Proposition \ref{n=3}.}
Note that every quadratic rational map, defined over $K$ with good reduction outside $S$, which admits a cycle in $\po$ of length $3$, is conjugate via an automorphism in ${\rm PGL}_2(K)$ to a map $\Psi$ of the shape as in (\ref {nf}), where $\lambda=1$, satisfying (\ref{pp0}).
In the proof of Proposition \ref{n=3} we shall use Lemma \ref{n=3part2} and the following lemma:
\begin{lemma}\label{n=3part1}Let $\Psi$ be a quadratic rational map defined over $K$ satisfying (\ref{pp0}). Let $[A]\in{\rm PGL}_2(K)$. If the rational map $[A]\circ\Psi\circ[A]^{-1}$ has good reduction outside $S$, then there exist two $S$-integers $a,c$ such that (\ref{nPsi3}) and part 1) of Proposition \ref{n=3} hold; furthermore:
\begin{itemize}
\item[i)]${\rm Disc}([A])^2=(c-1)\cdot R_S$,
\item[ii)]${\rm Disc}([A])$ divides the ideals $(1+a+a^2)\cdot R_S$ and $(1-(a+c)+(a+c)^2)\cdot R_S$.
\end{itemize}
\end{lemma}
\begin{proof}
As seen before in the proof of Lemma \ref{n=34}, up to enlarging $S$, we can suppose that the ring $R_S$ is a P.I.D..

Let $\Psi$ be a quadratic rational map satisfying (\ref{pp0}); then there exist three $S$-integers $a,b,c$ such that (\ref {nf}) holds with $\lambda=1$ and the hypotheses of Lemma \ref{n=34} are satisfied. Therefore by replacing $b$ by $1$, $a/b$ by $a$ and $c/b$ by $c$ in (\ref {nf}), we see that (\ref{nPsi3}) and part 1) of  Proposition \ref{n=3} are true.

By Lemma \ref{genA}, without loss of generality, we can assume that 
\beq\label{A} A=\begin{pmatrix}\alpha & 0\\ \beta & 1\end{pmatrix}\ \ \ \text{and}\ \ \ A^{-1}=\begin{pmatrix}1 & 0\\-\beta & \alpha\end{pmatrix}\eeq  for suitable $S$-integers $\alpha$ and $\beta$, where $\alpha$ and $\beta+1$ are coprime in $R_S$. 
We have that $ \Phi([X:Y])\coloneqq[A]\circ \Psi\circ [A]^{-1}=[F(X,Y):G(X,Y)]$; for the readers' convenience we rewrite the identities 
 (\ref{f2}) and (\ref{g2}) in the case $\lambda=1$ since they will be useful in the sequel, 
\begin{align}\label{f23}&F(X,Y)=\alpha(1+\beta)(a-\beta)X^2+\alpha^2(2\beta+1-a)XY-\alpha^3Y^2\\\label{g23}&G(X,Y)=[\beta(1+\beta)(a-\beta)+a-c\beta]X^2+\alpha(2\beta^2 +\beta -\beta a+c)XY-\alpha^2 \beta Y^2.\end{align}
As in the proof of Lemma \ref{n=34} we have that 
\begin{align}\label{N23}&\dfrac{(1+\beta)(a-\beta)}{\alpha}\in R_{S},\\
\label{D23}&\dfrac{\beta(1+\beta)(a-\beta)+a-c\beta}{\alpha^2}\in R_{S},\\
\label{D13}&\dfrac{2\beta^2 +\beta -\beta a+c}{\alpha}\in R_{S}.\end{align}
Since the  $S$-integers $\alpha$ and $1+\beta$ are coprime, by (\ref{N23}) it follows that 
\beq\label{NN23} \dfrac{a-\beta}{\alpha}\in R_{S};\eeq
since $a\in R_{S}^*$, $\alpha$ and $\beta$ are coprime.

The next step is to prove that $c-1=u\alpha^2$, for a suitable $S$-unit $u$, exploiting the good reduction outside $S$ of $\Phi$. 
Since $\Phi=[A]\circ \Psi\circ [A]^{-1}$ has good reduction outside $S$, from Definition \ref{gred} it follows that  $\vap({\rm Disc}(\Phi))=0$ for all prime ideals $\p\notin S$. We express ${\rm Disc}(\Phi)$ by using the representation $\Phi([X:Y])=[F(X,Y):G(X,Y)]$ with $F(X,Y),G(X,Y)$ as in  (\ref{f23}) and (\ref{g23}). By direct calculation we see that 
\beq\label{resphi} {\rm Res}(F,G)=-\alpha^6a(a+c)(c-1).\eeq
Otherwise, we can obtain (\ref{resphi}) by application of the following lemma:

\begin{lemma}\label{DeM} Let $U(X,Y)$ and $V(X,Y)$ be homogeneous polynomials of degree $D$, let $f(X,Y)$ and $g(X,Y)$ be homogeneous polynomials of degree $d$, and let 
\bdm B(X,Y)=U(f(X,Y),g(X,Y))\ \ \ \text{and}\ \ \ C(X,Y)=V(f(X,Y),g(X,Y))\edm
be their compositions. Then 
\bdm  {\rm Res}(B,C)= {\rm Res}(U,V)^d\cdot {\rm Res}(f,g)^{D^2}.\edm
\end{lemma}
\begin{proof}See Proposition 6.1 in \cite{DeM}
\end{proof}
By (\ref{nPsi3}) we have that $\Psi(X,Y)=[aX^2+(1-a)XY-Y^2:aX^2+cXY]$ and by direct calculation we see that
\bdm{\rm Res}(aX^2+(1-a)XY-Y^2,aX^2+cXY)=-a(a+c)(c-1).\edm
Moreover we have that $[A](X,Y)=[\alpha X:\beta X+Y]$ and ${\rm Res}(\alpha X,\beta X+Y)=\alpha$, 
$[A]^{-1}(X,Y)=[X:-\beta X+\alpha Y]$ and ${\rm Res}(X,-\beta X+\alpha Y)=\alpha$. Hence, by applying two times Lemma \ref{DeM} we obtain (\ref{resphi}).

Recall that, as said at the beginning of this proof, the hypotheses of Lemma \ref{n=34} are satisfied, so in particular $a\in R_{S}^*$. Hence, the $S$-integers $\alpha$ and $\beta$ are coprime because (\ref{NN23}) holds. 

Let $\beta\neq0$. By (\ref{f23}), (\ref{g23}), (\ref{N23}), (\ref{D23}), (\ref{D13}) and the coprimality of $\alpha$ and $\beta$, we see that $\min\{\vap(F), \vap(G)\}=2\vap(\alpha)$ for every prime ideal $\p\notin S$. 
Otherwise, if $\beta=0$, by (\ref{NN23}) and $a\in R_{S}^*$, it follows that $\alpha\in R_{S}^*$. Hence, for all prime ideals $\p\notin S$ we have $\vap(\alpha)=0$. Therefore by  (\ref{f23}) it follows that $\min\{\vap(F), \vap(G)\}=2\vap(\alpha)$ still holds and is equal to zero.
Hence, by Definition \ref{gred}, (\ref{resphi}), part 1) of  Proposition \ref{n=3}    and the good reduction of $\Phi$ outside $S$ we have that
\bdm 0=\vap({\rm Disc}(\Phi))=\vap({\rm Res}(F,G))-4\cdot\min\{\vap(F), \vap(G)\}=\vap(\alpha^6(c-1))-8\vap(\alpha)\edm
 holds for every $\p\notin S$. Now, we easily deduce that $c-1=u\alpha^2$, for a suitable $u\in R_{S}^*$. We have proved part i) of the lemma, since ${\rm Disc}([A])=\alpha\cdot R_S$.\\
Now, to prove part ii) we note that from (\ref{NN23}) and (\ref{D13}) the $S$-integer $\alpha$ divides $a^2+a+c$ and, since $\alpha\mid c-1$, it follows that $\alpha$ divides $1+a+a^2$ in $R_{S}$. To ease notation we set $v=a+c\in R_{S}^*$. Part i) of lemma reads as $u\alpha^2=v-a-1$, thus $a\equiv v-1$ (mod $\alpha$), so from $\alpha\mid 1+a+a^2$ in $R_{S}$ we easily obtain that $\alpha$ divides $1-v+v^2$ in $R_{S}$.
\end{proof}


\begin{proof}[Proof of Proposition \ref{n=3}]\noindent By Lemma \ref{n=3part1} it remains to prove only part 2). Without loss of generality we can take an enlarged finite set $S$ so that $R_S$ is a P.I.D.. Let $[A]\in {\rm PGL}_2(K)$ be an automorphism of $\pro$ such that the rational map $[A]\circ\Psi\circ[A]^{-1}$ has good reduction outside $S$. By Lemma \ref{genA} we can assume that the automorphism $[A]$ is induced by a matrix $A$ defined as in (\ref{A}). Let $a$ and $c$ be the $S$-integers satisfying (\ref{nPsi3}) and part 1).  Let $v$ be the $S$-unit $v=a+c$ and $\alpha$ an $S$-integer such that ${\rm Disc}([A])=\alpha\cdot R_S$. Lemma \ref{n=3part1} implies that there exists an $S$-unit $u$ such that the above $S$-units $a$ and $v$ satisfy the hypotheses of Lemma \ref{n=3part2}.  Thus, there exists a finite set $T$ containing $S$ such that  every  $S$-integer $\alpha$, as above, is a $T$-unit. By applying the $T$-unit Equation Theorem to the equation (\ref{alpha2}) we deduce that there are only finitely many possibilities for non-degenerate solutions $(a,v,u\alpha^2)\in R_{S}^*\times R_{S}^*\times R_{S}^{\phantom{*}}$ of (\ref{alpha2}), obtaining only finitely many possibilities for the ideal ${\rm Disc}([A])=\alpha\cdot R_S$.   Now it is clear that we can choose $\mathcal{B}$ equal to 
\footnotesize\begin{center} $\left\{a\in R_{S}^*\mid \text{ there exist }  v,u\in R_{S}^*\ \alpha\in R_{S},\ \text{ such that}\ \ (a,v,u\alpha^2) \ \text{\footnotesize is a non deg. sol. of (\ref{alpha2})} \right\}$\\
$\bigcup$\\
$\left\{v-a\mid a,v\in R_{S}^*\text{ and there exist}\ u\in R_{S}^*,\ \alpha\in R_{S},\ \text{such that}\ \ (a,v,u\alpha^2) \ \text{is a non deg. sol.  of (\ref{alpha2})} \right\}$,\end{center}\normalsize
which is a finite set.

We have degenerate solutions of (\ref{alpha2}) if and only if one of the following conditions hold:
\begin{itemize}
\item[i)] $a=-1$ and $u\alpha^2=v$;
\item[ii)] $v-a=0$ and $u\alpha^2=-1$;
\item[iii)] $v=1$ and $u\alpha^2=-a$,
\end{itemize}
which proves part 2) of Proposition \ref{n=3}. In all cases i), ii) and iii) we have that $\alpha\in R_{S}^*$. 
\end{proof}

\section{Proofs of the main theorems.}

\begin{proof}[Proof of Theorem \ref{m1}]By Lemma \ref{K=S} we can consider the conjugation action induced by ${\rm PGL}_2(K)$.  Every quadratic rational map, defined over $K$ with good reduction outside $S$ and admitting a periodic point in $\pro(K)$ with exact period $3$, is equivalent, via an automorphism of ${\rm PGL}_2(K)$, to a rational map $\Psi$ defined over $K$ which admits the following cycle 
\bdm[0:1]\stackrel{\Psi}{\mapsto}[1:0]\stackrel{\Psi}{\mapsto}[1:1]\stackrel{\Psi}{\mapsto}[0:1].\edm  

\noindent Proposition \ref{n=3} says that there are only finitely many possibilities for $\Psi$ or 
\begin{itemize}
\item[i)] $\Psi([X:Y])=[(X-Y)^2:X(X-(v+1)Y]$ where $v\in R_S^*$ or
\item[ii)] $\Psi([X:Y])=[(X-Y)(aX+Y):aX^2]$ where $a\in R_S^*$ or
\item[iii)] $\Psi([X:Y])=[(X-Y)(aX+ Y  ):X(aX+(1-a)Y]$ where $a\in R_S^*$.
\end{itemize}
Let 
\bdm B=\begin{pmatrix}1& -1\\1&0\end{pmatrix}.\edm
The induced automorphism $[B]=[X-Y:X]\in {\rm PGL}_2(K)$ has inverse $[B]^{-1}=[-Y:X-Y]$. By direct calculation we see that 
\bdm [B]\circ [(X-Y)^2:X(X-(v+1)Y]\circ[B]^{-1}=[(X-Y)(-v^{-1}X+Y):-v^{-1}X^2]\edm which has the form as in ii) and 
\bdm [B]^{-1}\circ [(X-Y)(aX+ Y  ):X(aX+(1-a)Y]\circ[B]=[(X-Y)(-a^{-1}X+Y):-a^{-1}X^2]\edm which still has the form as in ii). This concludes the proof.
\end{proof}

Now we prove that the set of maps of the form as in (\ref{param}) represents an infinite set of classes under the conjugation action induced by ${\rm PGL}_2(\overline{\mathbb{Q}})$.

\begin{prop}\label{endp}For every $a\in\mathbb{G}_m$, let $\Phi_a$ be the quadratic map defined in (\ref{param}). Each element of $M_2(\overline{\mathbb{Q}})$ contains at most six maps of the form $\Phi_a$.
\end{prop} 
\begin{proof}[Proof of Proposition \ref{endp}]Let $a\in\mathbb{G}_m$ be fixed. 
We have to prove that there exist at most six maps of the form $\Phi_b$  conjugate    to the map $\Phi_a$. Notice that the rational map $\Phi_a$ admits at most two cycles of length $3$ in $\pro(\overline{\mathbb{Q}})$. The ordered triple $([0:1],[1:0],[1:1])$ is a cycle for $\Phi_a$ for every integer $a$. Let $(P_0,P_1,P_2)$ be the \vvs second\vvd\ cycle of length $3$ for $\Phi_a$ (it could be equal to $([0:1],[1:0],[1:1])$). If $[A]\in{\rm PGL}_2(\overline{\mathbb{Q}})$ is an automorphism such that the map $[A]\circ\Phi_a\circ[A]^{-1}=\Phi_b$ for some $b\in\mathbb{G}_m$, then either 
\bdm ([A]([0:1]),[A]([1:0]),[A]([1:1]))=([0:1],[1:0],[1:1])\edm
or 
\bdm ([A](P_0),[A](P_1),[A](P_2))=([0:1],[1:0],[1:1]).\edm
Hence we deduce that there are only six possibilities for the automorphism $[A]\in{\rm PGL}_2(\overline{\mathbb{Q}})$ and so there are only six possibilities for the map $\Phi_b$. \end{proof}
To prove Theorem \ref{tf32} we shall use the following result:
\begin{prop}\label{fic}
Let $\mathbb{I}$ be a given finite set of ideals of $R_S$. Let $n$ be a fixed positive integer. There exist only finitely many inequivalent tuples $\ciclon$ such that 
\beq\label{fec}\mathfrak{I}(P_i,P_j)\in\mathbb{I}\ \ \ \text{for all distinct $0\leq i,j\leq n-1$},\eeq
where the ideals $\mathfrak{I}(P_i,P_j)$ are those defined in (\ref{I_ij}).
\end{prop}
This proposition is a generalization of Proposition 2 in \cite{C.1} where the set $\mathbb{I}$ contains only the ideal $R_S$. Actually the techniques to prove these two propositions are the same.
Therefore we give here a very brief proof, for the details see the proof in \cite{C.1}. 
\begin{proof}
Let $(P_0,P_1,\ldots,P_{n-1})$ be a tuple which satisfies (\ref{fec}). For every index $0\leq i\leq (n-1)$, since ${\mathbb{I}}$ is a finite set of ideals, by the finiteness of the class number of $R_S$, we can choose  $(x_i,y_i)\in R_S^2$ as homogeneous coordinates representing the point $P_i=[x_i:y_i]$, such that for all $0\leq i,j\leq (n-1)$ the equalities $x_iy_j-x_jy_i=r_{i,j}u_{i,j}$ hold, where $u_{i,j}\in R_S^*$ and $r_{i,j}$ belong to a finite set which can be chosen only depending on $S$ and ${\mathbb{I}}$. 
To every tuple $(P_0,P_1,\ldots,P_{n-1})$, and for every index $0\leq i\leq (n-1)$, we associate the following binary form
\bdm F(X,Y)=\prod_{0\leq i\leq (n-1)}(x_iX-y_iY)\edm
defining in this way a family of forms with discriminant 
\bdm D(F)=u\left(\prod_{0\leq i<j\leq (n-1)}r_{i,j}^2\right)\edm
where $u$ is a suitable $S$-unit.

There exists a finite set $W$, only depending on $S$, such that for every $S$-unit $u$ we have that $u=w\cdot v^{2n-2}$ for suitable $w\in W$ and $v\in R_S^*$. Therefore, for every cycle $(P_0,P_1,\ldots,P_{n-1})$ satisfying (\ref{fec}) there exists a $S$-unit $v^{-1}$ such that by taking $(v^{-1}\cdot x_0,v^{-1}\cdot y_0)$ as homogeneous coordinates  representing the point $P_0=[x_0:y_0]$, we can assume that the discriminant of every binary form $F(X,Y)$ defined above belongs to a finite set which only depends on $S$ and ${\mathbb{I}}$.
Now it is sufficient to apply the result obtained by Birch and Merriman in \cite{B.M.1} on the finiteness of the classes of binary forms with given degree and given discriminant. See \cite{E.4} for an effective form of the Birch and Merriman's result.
\end{proof}
\begin{proof}[Proof of Theorem \ref{tf32}]%
The case $n\geq 4$ is an immediate application of Theorem $\ref{tnf}^\prime$.

Let $n=3$. Let $(P_0,P_1,P_2)$ be a cycle in $\po$ for a quadratic rational map $\Phi$ defined over $K$ with good reduction outside $S$. 
Let $[C]\in{\rm PGL}_2(K)$ an automorphism which sends the ordered triple $(P_0,P_1,P_2)$ into the ordered triple \mbox{$([0:1]$},\mbox{$[1:0]$},\mbox{$[1:1])$}. The ordered $n$-tuple $([0:1],[1:0],[1:1])$ is a cycle for the quadratic rational map $[C]\circ\Phi\circ[C]^{-1}$. We set $\Psi=[C]\circ\Phi\circ[C]^{-1}$ and $[A]=[C]^{-1}$. The rational map $\Psi$ and the automorphism $[A]$ satisfy the hypothesis of Proposition \ref{n=3}. Furthermore $[A]([0:1])=P_0$ and $[A]([1:0])=P_1$. Let
\bdm A=\begin{pmatrix}a&b\\c&d\end{pmatrix}\edm
be a matrix in M$_{2\times 2}(R_S)$ which induces the automorphism $[A]$. We see immediately that 
$\vap({\rm Disc}(A))=\vap(ad-cb)-2\min\{\vap(a),\vap(b),\vap(c),\vap(d)\}$
and $P_0=[b:d]$ and $P_1=[a:c]$ so that
\bdm \vap(\mathfrak{I}_1)=\vap(\mathfrak{I}(P_0,P_1))=\vap (ad-cb)-\min\{\vap(a),\vap(c)\}-\min\{\vap(b),\vap(d)\}.\edm
In this way we have proved that $\vap(\mathfrak{I}_1)\leq\vap({\rm Disc}(A))$. Since Proposition \ref{n=3} tells us that there are only finitely many possibilities for the ideal ${\rm Disc}(A)$, the same holds also for the ideal $\mathfrak{I}_1$. By \cite[Proposition  6.1]{M.S.2} we have that $\mathfrak{I}_1=\mathfrak{I}_2$, hence there exists a finite set ${\mathbb{I}}$ such that if $(P_0,P_1,P_2)$ is a cycle for a quadratic rational map defined over $K$ with good reduction outside $S$, then 
\bdm\mathfrak{I}(P_i,P_j)\in{\mathbb{I}}\ \ \ \text{for all distinct $0\leq i\neq j\leq 2$}.\edm
By applying Proposition \ref{fic} we obtain that the set of classes of $3$-cycles, for some quadratic rational maps defined over $K$ with good reduction outside $S$, is finite. 
\end{proof}

\bibliographystyle{plain}
\bibliography{DynamicalSystems}

\bigskip

\noindent Jung Kyu CANCI\\
e-mail: jungkyu.canci@math.univ-lille1.fr\\
Laboratoire Paul Painlevé, Mathématiques, Université Lille 1, 59655 Villeneuve d'Ascq Cedex, France
\end{document}